\documentclass{amsart}

\usepackage{amsfonts}
\usepackage{latexsym}
\usepackage{amsmath}
\usepackage{amssymb}

\newcommand{\C}{\mathbb{C}}
\newcommand{\N}{\mathbb{N}}
\newcommand{\Z}{\mathbb{Z}}
\renewcommand{\P}{\mathbb{P}}
\newcommand{\card}{\mathrm{card}}
\newcommand{\wdn}{\underline w}
\newcommand{\wu}{\overline w}

\newcommand{\sn}{\textrm{sn}}

\newtheorem{theorem}{Theorem}
\newtheorem{lemma}[theorem]{Lemma}
\newtheorem{corollary}[theorem]{Corollary}

\theoremstyle{definition}

\newtheorem{example}[theorem]{Example}

\newtheorem{remark}[theorem]{Remark}

\numberwithin{equation}{section}
\numberwithin{theorem}{section}

\begin{document}

\title{Holomorphic curves with shift-invariant hyperplane preimages}

\author{Rodney Halburd}
\address{Department of Mathematics, University College London,
Gower Street, London WC1E 6BT, UK} \email{r.halburd@ucl.ac.uk}
\author{Risto Korhonen}
\address{Department of Physics and Mathematics, University of Eastern Finland, P.O. Box 111,
FI-80101 Joensuu, Finland}
\email{risto.korhonen@uef.fi}
\author{Kazuya Tohge}
\address{Graduate School of Natural Science and Technology, Kanazawa University, Kakuma-machi, Kanazawa, 920-1192, Japan}
 \email{tohge@t.kanazawa-u.ac.jp}
\thanks{Supported in part by the Academy of Finland Grant \#112453, \#118314 and \#210245, a grant from the EPSRC, the Isaac Newton Institute for Mathematical Sciences, the Japan Society for the Promotion of Science Grant-in-Aid for Scientific Research (C) \#19540173, and a project grant from the Leverhulme Trust.}
\subjclass[2000]{Primary 32H30, Secondary 30D35}

\keywords{Holomorphic curve, Casorati Determinant, Difference operator, Borel's theorem, Nevanlinna theory, Cartan's Second Main Theorem}

\begin{abstract}
If $f:\C\to\P^n$ is a holomorphic curve of hyper-order less than
one for which $2n+1$ hyperplanes in general position have forward invariant preimages with respect to the
translation $\tau(z)= z+c$, then $f$ is periodic with
period~$c\in\C$. This result, which can be described as a
difference analogue of M.~Green's Picard-type theorem for
holomorphic curves, follows from a more general result presented
in this paper. The proof relies on a new version of Cartan's
second main theorem for the Casorati determinant and an extended
version of the difference analogue of the lemma on the logarithmic
derivatives, both of which are proved here. Finally, an
application to the uniqueness theory of meromorphic functions is
given, and the sharpness of the obtained results is demonstrated
by examples.
\end{abstract}

\maketitle

\section{Introduction}

According to Picard's theorem all holomorphic mappings
$f:\C\to\P^1\setminus{\{a_1,a_2,a_3\}}$ are constants. For holomorphic
curves in $\P^n$ where $n\geq 2$ Bloch \cite{bloch:26} and Cartan
\cite{cartan:28} showed that if a non-constant holomorphic mapping
$f:\C\to\P^n$ misses $n+2$ hyperplanes in general
position, then the image of $f$ lies in a proper linear subspace of $\P^n$.
Here a \textit{hyperplane} $H$ is the set of all points $x\in\P^n$,
$x=[x_0:\cdots:x_n]$, such that
    \begin{equation}\label{hyperplanedef}
    \alpha_0 x_0+\cdots+\alpha_n x_n = 0,
    \end{equation}
where $\alpha_j\in \C$ for
$j=0,\ldots,n$. The hyperplanes
$H_k$, $k=0,\ldots,m$, defined by $\alpha_{0,k}
x_0+\cdots+\alpha_{n,k} x_n = 0$  are said to be in
\textit{general position} if $m\geq n$ and any $n+1$ of the
vectors $\alpha_k=(\alpha_{0,k},\ldots,\alpha_{n,k})\in \C^{n+1}$
are linearly independent.

Another natural generalization of Picard's theorem was given by
Fujimoto \cite{fujimoto:72b} and Green \cite{green:72}, who showed
that if $f:\C\to\P^n$ omits $n+p$ hyperplanes in general position where $p\in\{1,\ldots,n+1\}$, then the image of $f$ is contained in a linear subspace of dimension at most $[n/p]$. In particular, by taking $p=n+1$ it follows
that if the image of a holomorphic function $f:\C\to\P^n$ lies in
the complement of $2n+1$ hyperplanes in general position, then $f$
must be a constant. Further extensions of Picard's theorem for
holomorphic curves missing hyperplanes can be found, for instance,
in \cite{fujimoto:74,green:74,green:75}.

We say that the preimage of a hyperplane $H\subset\P^n$ under
$f$ is \textit{forward invariant} with respect to the translation
$\tau(z)= z+c$ if $\tau(f^{-1}(\{H\}))\subset f^{-1}(\{H\})$ where
$f^{-1}(\{H\})$ and $\tau(f^{-1}(\{H\}))$ are multisets in which
each point is repeated according to its multiplicity. Let, for instance,
$\varphi(z)$ be an entire function given by the pullback divisor
of the hyperplane~$H$. If
	\begin{equation*}
	\varphi(z)=\frac{\varphi^{(i)}(z_0)}{i!}(z-z_0)^i+O((z-z_0)^{i+1}),\quad \varphi^{(i)}(z_0)\not=0,
	\end{equation*}	
and
	\begin{equation*}
	\varphi(z+c)=\frac{\varphi^{(j)}(z_0)}{j!}(z-z_0)^j+O((z-z_0)^{j+1}),\quad \varphi^{(j)}(z_0)\not=0,
	\end{equation*}	
for all $z$ in a neighborhood of $z_0$ and $j\geq i>0$, the point $z_0$ is a
forward invariant element in a preimage of $H$ with respect
to $\tau(z)$, while if $i> j$ then $z_0$ is not a forward
invariant element. By this definition the preimages of omitted
hyperplanes of $f$ are special cases of forward
invariant preimages since in this case $f^{-1}(\{H\})=\emptyset$.
One of the purposes of this paper is to show that analogous
results to Picard's theorem for holomorphic curves $f:\C\to\P^n$
can be obtained even if the image of $f$ intersects with the
target hyperplanes in general position, provided that at the same
time the preimages of these hyperplanes under $f$ are forward invariant with respect to a translation, and the considered holomorphic curve does not grow too fast.

The growth is characterized by the means of Nevanlinna theory in the following way. The \textit{order of growth} of a holomorphic curve $f:\C\to\P^n$
with homogeneous coordinate $f=[f_0:\cdots:f_n]$ is defined by
    \begin{equation}\label{order}
    \sigma(f)=\limsup_{r\to\infty}\frac{\log^+ T_f(r)}{\log r},
    \end{equation}
where $\log^+x=\max\{0,\log x\}$ for all $x\ge0$, and
    \begin{equation}\label{characteristic}
    T_f(r):=\int_0^{2\pi}u(re^{i\theta})\frac{d\theta}{2\pi}-u(0),
    \quad u(z)=\sup_{k\in \{0,\ldots,n\}} \log|f_k(z)|,
    \end{equation}
is the \textit{Cartan characteristic function} of $f$. Note here
that this representation of $f$ is to be reduced in the sense that
the $n+1$ functions $f_j$ are entire functions without common
zeros. The \textit{hyper-order} of a holomorphic curve $f:\C\to\P^n$ is defined by
    \begin{equation}\label{horder}
    \varsigma(f)=\limsup_{r\to\infty}\frac{\log^+\log^+ T_f(r)}{\log
    r}
    \end{equation}
and the usual Nevanlinna hyper-order is
    \begin{equation*}
    \rho_2(w)=\limsup_{r\to\infty}\frac{\log^+\log^+T(r,w)}{\log r},
    \end{equation*}
where $w$ is meromorphic in the complex plane and $T(r,w)$ is the Nevanlinna characteristic function of $w$. Since, by writing $w=[w_0:w_1]$ where $w_0$ and $w_1$ are entire functions without common zeros, it follows that $\rho_2(w)=\varsigma(w)$, we will use the notation $\varsigma(w)$ from now on to denote the hyper-order of the meromorphic function~$w$.

Let $c\in\C$, and let $\mathcal{P}^1_c$ be the field of period $c$ meromorphic functions
defined in $\C$ of hyper-order strictly less than one. The following theorem is a difference analogue of Picard's theorem for holomorphic curves.

\begin{theorem}\label{hypers}
Let $f:\C\to\P^n$ be a holomorphic curve such that
$\varsigma(f)<1$, let $c\in\C$ and let $p\in\{1,\ldots,n+1\}$. If $n+p$ hyperplanes in general position have forward invariant preimages under $f$ with respect to the translation
$\tau(z)= z+c$, then the image of $f$ is contained in a projective
linear subspace over $\mathcal{P}^1_c$ of dimension $\leq[n/p]$.
\end{theorem}

The following example shows that the growth condition $\varsigma(f)<1$ in Theorem~\ref{hypers} cannot be replaced by $\varsigma(f)\leq1$.

\begin{example}
Put $\omega=2\pi /(2\log 6)$ and consider a linearly non-degenerate holomorphic curve $f:\mathbb{C}\to\mathbb{P}^3$ given by
$$
f:=[-\sin^2 \omega z: -\cos^2 \omega z: (\sin^2 \omega z) \exp e^z : (\cos^2 \omega z) \exp e^z],
$$
which is not $(2\log 6)$-periodic and has the hyper-order $\varsigma(f)=1$.
Take the following seven hyperplanes located in general position in $\mathbb{P}^3$:
$$
H_1=\{w\, |\, h_1(w):=w_0=0\}\,,
$$
$$
H_2=\{w\, |\, h_2(w):=w_1=0\}\,,
$$
$$
H_3=\{w\, |\, h_3(w):=w_2=0\}\,,
$$
$$
H_4=\{w\, |\, h_4(w):=w_3=0\}\,,
$$
$$
H_5=\{w\, |\, h_5(w):=w_0+w_1+w_2+w_3=0\}\,,
$$
$$
H_6=\{w\, |\, h_6(w):=w_0+\eta_5w_1+\eta_5^2w_2+\eta_5^3w_3=0\}\,,
$$
$$
H_7=\{w\, |\, h_7(w):=w_0+\eta_7w_1+\eta_7^2w_2+\eta_7^3w_3=0\}\,,
$$
where $w=[w_0:w_1:w_2:w_3]$ and $\eta_5$ and $\eta_7$ are the primitive fifth and seventh root of unity,
respectively.
Then we have
$$
h_1(f)= -\sin^2 \omega z\,, \quad h_2(f)= -\cos^2 \omega z\,,
$$
$$
h_3(f)= (\sin^2 \omega z) \exp e^z\,, \quad h_4(f)= (\cos^2 \omega z) \exp e^z,
$$
all of whose zero preimages are forward invariant with respect to $\tau(z)=z+2\log 6$, while
$$
h_5(f)= (\sin^2 \omega z + \cos^2 \omega z ) (\exp e^z-1) \,,
$$
$$
h_6(f)= \eta_5^2(\sin^2 \omega z + \eta_5 \cos^2 \omega z ) (\exp e^z - \eta_5^3) \,,
$$
$$
h_7(f)= \eta_7^2(\sin^2 \omega z + \eta_7 \cos^2 \omega z ) (\exp e^z - \eta_7^5) \,,
$$
each of whose zeros have forward invariant preimages with respect to $\tau(z)=z+2\log 6$, or are points such that $\exp e^z =\alpha$ for some $\alpha\in\{1,\eta_5^3, \eta_7^5\}$.
Then $\alpha$ is a $35^\textrm{th}$ root of unity and thus all preimages of these hyperplanes are forward invariant with respect to $\tau(z)=z+2\log 6$. On the other hand, the image of $f$ is contained in a projective linear subspace over $\mathcal{P}_{2\log6}^1$ of dimension $1$ (even though $f$ is linearly non-degenerate in the usual sense) but `$[n/p]$' in Theorem~\ref{hypers} is $[n/p]=[3/(7-3)]=0$. In fact the image is on the projective line described by the two hyperplanes
$(\cos^2\omega z)w_0-(\sin^2\omega z)w_1=0$ and $(\cos^2\omega z)w_2-(\sin^2\omega z)w_3=0$
over $\mathcal{P}_{2\log6}^1$, and also it does not degenerate into a singleton in the space, since $\exp e^z \not\in \mathcal{P}_{2\log6}^1$.
\end{example}

An example demonstrating the sharpness of the upper bound $[n/p]$ in Theorem~\ref{hypers} is given in section \ref{examplesec} below.
The following corollary is immediately obtained by applying Theorem~\ref{hypers}
with $p=n+1$.

\begin{corollary}\label{hyperscor}
Let $f:\C\to\P^n$ be a holomorphic curve such that
$\varsigma(f)<1$, and let $c\in\C$. If $2n+1$ hyperplanes in general position have forward invariant preimages under $f$ with respect to
the translation $\tau(z)= z+c$, then $f$ is periodic with
period~$c$.
\end{corollary}

If the preimage of a hyperplane under a holomorphic curve
$f:\C\to\P^n$ is empty, then it is clearly forward invariant with
respect to all translations of the complex plane. Therefore, if
$f$ omits $2n+1$ hyperplanes in general position, then it follows
by Corollary~\ref{hyperscor} that $f$ is, in fact, a periodic
holomorphic curve with all periods $c\in\C$. This is, of course,
only possible when $f$ is a constant function. We have just shown
that Corollary~\ref{hyperscor} implies M.~Green's Picard-type
theorem for holomorphic curves \cite{green:72} in the special case
$\varsigma(f)<1$.

A simple example shows that the growth condition $\varsigma(f)<1$
in Corollary~\ref{hyperscor} cannot
be significantly weakened. For $f(z)=[\exp(\exp(z)):1]:\C\to\P^1$
each of the $n^{\textrm{th}}$ roots of unity $[1:-\exp(2 m\pi
i/n)]$, $m\in\{1,\ldots,n\}$, has a forward invariant preimage
with respect to $\tau(z)= z+\log(n+1)$, but nevertheless
$f(z)\not\equiv f(z+\log(n+1))$. Therefore $f$ is an example of a
holomorphic curve which has arbitrarily many target values with
forward invariant preimages, even though it just barely fails to
satisfy the condition $\varsigma(f)<1$.

Finite-order meromorphic solutions of difference equations have
been under careful study recently.
Ruijsenaars has been studying minimal solutions of certain classes
of linear difference equations as part of a programme of
developing Hilbert space theory for analytic difference operators
\cite{ruijsenaars:97,ruijsenaars:00}. In the nonlinear case,
Ablowitz, Halburd and Herbst \cite{ablowitzhh:00} suggested that
the existence of sufficiently many finite-order meromorphic
solutions can be used to detect difference equations of Painlev\'e
type. Difference quotient estimates
\cite{halburdk:06JMAA,halburdk:06AASFM,chiangf:08,chiangf:09} have proved to be useful tools in much
of the recent analysis involving finite-order meromorphic solutions of difference equations (see, e.g. \cite{chiangr:06,halburdk:07PLMS,halburdk:07JPA,lainey:07}) but so far there is limited amount of information available on the
behavior of fast growing solutions. Another main purpose of this paper is to show that if $f$ is a meromorphic function such that $\varsigma(f)=\varsigma<1$ and $\varepsilon>0$, then
    \begin{equation}\label{DAintro}
    m\left(r,\frac{f(z+c)}{f(z)}\right)=o
    \left(\frac{T(r,f)}{r^{1-\varsigma-\varepsilon}}\right)
    \end{equation}
for all $r$ outside of a set of finite logarithmic measure (see Theorem~\ref{logder} below). The type of difference analogue of the lemma on the logarithmic derivatives represented by \eqref{DAintro} cannot be in general extended to meromorphic functions of hyper-order at least one, since $g(z)=\exp(2^z)$ satisfies $g(z+1)/g(z)=g(z)$, and so $m(r,g(z+1)/g(z))=T(r,g)$.

The remainder of the paper is organized in the following way.
Section~\ref{cartansec} contains a difference analogue of Cartan's
generalization of the second main theorem of Nevanlinna theory, which will be applied in section~\ref{borelsec} to obtain a difference analogue of Borel's theorem on linear combinations of entire functions without zeros. These results are some of the main components in the proof of
Theorem~\ref{hypers} in section~\ref{section_proofh}.  Applications
of these results to uniqueness theory of meromorphic functions are
discussed in section~\ref{uniquenesssec}. The proof of the
difference Cartan in section~\ref{section_proof} relies on a
logarithmic difference estimate given in section~\ref{logdesec},
and proved in section~\ref{section_lemma}. A discussion on
$q$-difference analogues of the above results is given in
section~\ref{section_q}, and, finally, the sharpness of some of
the main results is considered in section~\ref{examplesec}.

\section{Difference analogue of Cartan's second main theorem}\label{cartansec}

The second main theorem of Nevanlinna theory \cite{nevanlinna:25}
is a deep generalization of Picard's theorem for meromorphic
functions in the complex plane, and a cornerstone on which the
whole value distribution theory lies. Cartan's version of
the second main theorem \cite{cartan:33} is a generalization of this result  to holomorphic curves \cite{lang:87}, and it has also turned out to be a useful tool for certain problems in the complex plane, for instance, in considering Waring's problem for analytic functions \cite{hayman:84} and unique range sets for entire functions \cite{gundersen:03,gundersenh:04}.

We now recall some of the known properties of the Cartan
characteristic function from \cite{gundersenh:04,lang:87}. For
instance, if $g=[g_0:\cdots:g_n]$ with $n\geq 1$ is a reduced
representation of a non-constant holomorphic curve $g$, then $T_g(r)\to\infty$ as $r\to\infty$, and
if at least one quotient $g_j/g_m$ is a transcendental function,
then $T_g(r)/\log r\to\infty$ as $r\to\infty$. Moreover, if
$f_0,\ldots,f_q$ are $q+1$ linear combinations of the functions
$g_0,\ldots,g_n$ over $\C$, where $q>n$, such that any $n+1$ of
the $q+1$ functions $f_0,\ldots,f_q$ are linearly independent,
then
    \begin{equation}\label{cartanT}
    T\left(r,\frac{f_\mu}{f_\nu}\right)\leq T_g(r) + O(1)
    \end{equation}
where $r\to\infty$, and $\mu$ and $\nu$ are distinct integers in
the set $\{0,\ldots,q\}$. Moreover, if $n=1$, then \eqref{cartanT}
becomes an asymptotic identity.

The order of a holomorphic curve $f:\C\to\P^n$ is independent of the
reduced representation of $f$. For if $[f_0:\cdots:f_n]$ and
$[F_0:\cdots:F_n]$ are two reduced representations of the curve
$f$, then, since the $f_j$'s and $F_j$'s are entire and
    \begin{equation*}
    \max_{j=0,\ldots,n}|f_j(z)|\not=0 \,\textrm{ and }\,
    \max_{j=0,\ldots,n}|F_j(z)|\not=0,
    \end{equation*}
it follows that there exists a nowhere vanishing entire function
$h$ such that
    \begin{equation*}
    F_j(z)=h(z)f_j(z)
    \end{equation*}
for all $z\in\C$ and $j\in\{0,\ldots,n\}$. By writing
$F=[F_0:\cdots:F_n]$ and defining
    \begin{equation*}
    T_F(r)=\int_0^{2\pi}U(re^{i\theta})\frac{d\theta}{2\pi}-U(0),
    \quad U(z)=\sup_{k\in \{0,\ldots,n\}} \log|F_k(z)|,
    \end{equation*}
it follows that
    \begin{equation*}
    T_F(r)=T_f(r)+\frac{1}{2\pi}\int_0^{2\pi}\log|h(re^{i\theta})|d\theta-\log|h(0)|.
    \end{equation*}
However, since $h(z)$ is entire and nowhere zero, it follows that
$\log|h(z)|$ is harmonic, and therefore
    \begin{equation*}
    \log|h(0)|=\frac{1}{2\pi}\int_0^{2\pi}\log|h(re^{i\theta})|d\theta.
    \end{equation*}
Hence $T_F(r)=T_f(r)$ is independent of the representation of $f$
in terms of projective coordinates, and so the order of $f$ is
well defined by \eqref{order}. We refer to \cite{lang:87} for the
full description of Cartan's value distribution theory, and
\cite{cherryy:01,hayman:64,yang:93} for the standard notation of
Nevanlinna theory.

Let $g(z)$ be a meromorphic function, and let $c\in\C$. We will
use the short notation
    \begin{equation*}
    g(z)\equiv g, \quad g(z+c)\equiv \overline{g},\quad g(z+2c)\equiv
    \overline{\overline{g}}\quad \textrm{and} \quad g(z+nc)\equiv
    \overline{g}^{[n]}
    \end{equation*}
to suppress the $z$-dependence of $g(z)$. The \textit{Casorati
determinant} of $g_0,\ldots,g_n$ is then defined by
    \begin{equation*}
    C(g_0,\ldots,g_n)=\left|\begin{array}{ccccc}
      g_0 & g_1 & \cdots & g_n \\
      \overline{g}_0 & \overline{g}_1 & \cdots & \overline{g}_n \\
      \vdots & \vdots & \ddots & \vdots \\
      \overline{g}^{[n]}_0 & \overline{g}^{[n]}_1 &  \cdots & \overline{g}^{[n]}_n \\
    \end{array}\right|.
    \end{equation*}

In Cartan's generalization of the second main theorem the
ramification term is expressed in terms of the Wronskian
determinant of a set of linearly independent entire functions. The
following theorem is a difference analogue of Cartan's result
where the ramification term has been replaced by a quantity
expressed in terms of the Casorati determinant of functions which
are linearly independent over a field of periodic functions.

\medskip

\noindent\textbf{Note added in proof. } After this paper was completed we learnt of the paper \cite{wonglw:09} by Pit-Mann Wong, Hiu-Fai Law and Philip P. W. Wong, where a result which is very similar to Theorem \ref{casorati} was obtained. The distinct difference between the two results is found in the assumption on the growth order of a holomorphic curve under consideration. In \cite{wonglw:09} it is assumed that the curve is given by a reduced representation whose components are all entire functions of finite order, while a case of infinite order is permitted in this paper. The difference has caused different choices of an auxiliary function which is essential in the two papers, which might be observed by the expressions $L$ in \eqref{H} and $\tilde{L}$ in \eqref{Ltilde} here in this paper.

\medskip

\begin{theorem}\label{casorati}
Let $n\geq 1$, and let $g_0,\ldots,g_n$ be entire functions,
linearly independent over $\mathcal{P}^1_c$, such that
$\max\{|g_0(z)|,\ldots,|g_n(z)|\}>0$ for each $z\in\C$, and
    \begin{equation}\label{growth}
    \varsigma:=\varsigma(g)<1,\quad
    g=[g_0:\cdots:g_n].
    \end{equation}
Let $\varepsilon>0$. If $f_0,\ldots,f_q$ are $q+1$ linear
combinations of the $n+1$ functions
$g_0,\ldots,g_n$, where $q>n$, such that any $n+1$ of the $q+1$
functions $f_0,\ldots,f_q$ are linearly independent, and
    \begin{equation}\label{H}
    L=\frac{f_0f_1\cdots
    f_q}{C(g_0,g_1,\ldots,g_n)},
    \end{equation}
then
    \begin{equation}\label{casoratiineq}
    (q-n)T_g(r)\leq
    N\left(r,\frac{1}{L}\right)-N(r,L)+o\left(\frac{T_g(r)}{r^{1-\varsigma-\varepsilon}}\right)+O(1),
    \end{equation}
where $r$ approaches infinity outside of an exceptional set $E$ of finite
logarithmic measure (i.e. $\int_{E\cap[1,\infty)} dt/t <\infty$).
\end{theorem}

In \cite{halburdk:06AASFM} an analogue of the second main theorem
for the difference operator $\Delta_c f=f(z+c)-f(z)$ was introduced.
We will now show that, for constant targets, Theorem~\ref{casorati} is a generalization
of this result in a similar way as Nevanlinna's second main
theorem follows by Cartan's result.

Let $w$ be a meromorphic function such that the usual Nevanlinna
hyper-order satisfies $\varsigma(w)<1$. Then there exist linearly independent
entire functions $g_0$ and $g_1$ with no common zeros such that
$w=g_0/g_1$, and, according to \eqref{cartanT}, it follows that
$\varsigma(g)<1$ for $g=[g_0:g_1]$. Note that in general the
entire functions $g_0$ and $g_1$ themselves may be of hyper-order greater or
equal to one (see
\cite{bergweiler:94}).

Let $a_j\in\C$ for $j=0,\ldots,q-1$, and denote $f_j=g_0-a_j g_1$
and $f_{q}=g_1$. Then, by Theorem \ref{casorati}, it follows that
    \begin{equation}\label{applT}
    (q-1)T_g(r)\leq N\left(r,\frac{1}{L}\right)-N(r,L)+o(T_g(r))
    \end{equation}
where
    \begin{equation*}
    L=\frac{f_0f_1\cdots
    f_{q-1}g_{1}}{g_0\overline{g}_1-\overline{g}_0
    g_1}.
    \end{equation*}
We define the counting function $\widetilde{N}$ for $a\in\C$ as in
\cite{halburdk:06AASFM} by
    \begin{equation}\label{c1}
    \widetilde{N}\left(r,\frac{1}{w-a}\right)=\int_0^r\frac{\widetilde{n}(t,a)-\widetilde{n}(0,a)}{t}\,dt+
    \widetilde{n}(0,a)\log r
    \end{equation}
where $\widetilde{n}(r,a)$ counts the number of $a$-points of $w$
with multiplicity of $w(z_0)=a$ counted according to multiplicity
of $a$ at $z_0$ \textit{minus} the order of the (possible) zero of
$\Delta_c w$ at $z_0$. The pole counting function is then
    \begin{equation}\label{c2}
    \widetilde{N}(r,w)=\widetilde{N}\left(r,\frac{1}{1/w}\right).
    \end{equation}
By interpreting \eqref{applT} in terms of the counting
functions \eqref{c1} and \eqref{c2}, and using \eqref{cartanT} we
have
    \begin{equation*}
    (q-1)T(r,w)\leq \widetilde{N}(r,w)+ \sum_{j=0}^{q-1}
    \widetilde{N}\left(r,\frac{1}{w-a_j}\right)-N_0\left(r,\frac{1}{\Delta_c w}\right)
    +o(T(r,w))
    \end{equation*}
where $N_0(r,1/\Delta_c w)$ counts the number of those zeros of
$\Delta_c w$ which do not coincide with any of the $a_j$-points or
poles of $w$, and $r$ runs to infinity outside of a set of finite
logarithmic measure. This is an extension of \cite[Theorem
2.5]{halburdk:06AASFM} as desired.

\section{Difference analogue of Borel's theorem}\label{borelsec}

According to Borel's theorem, if $h_0,\ldots,h_n$ are entire functions without zeros, then the only possible solutions of the equation
	\begin{equation}\label{boreleqh}
	h_0+\cdots+h_n=0
	\end{equation}
are trivial solutions of the form
	$$
	h_0+\cdots+h_n=\sum_{k=1}^l\sum_{i\in S_k} c_{i,j_k} h_{j_k},
	$$
where $S_k$, $k=1,\ldots,l$, is the partition of $\{0,\ldots,n\}$ formed so that $i$ and $j$ are in $S_k$ if and only if $h_i/h_j\in \C$, and  $\sum_{i\in S_k} c_{i,j_k}=0$ for all $k=1,\ldots,l$ (see, e.g. \cite[p.~186]{lang:87} or \cite[p.~124]{ru:01}). The following difference analogue of Borel's theorem will be one of the key results needed in the proof of Theorem~\ref{hypers}.

\begin{theorem}\label{borelanalogue}
Let $c\in\C$, and let $g=[g_0:\cdots:g_n]$ be a holomorphic curve such that $\varsigma(g)<1$ and such that preimages of all zeros of $g_0,\ldots,g_n$ are forward invariant with respect to the translation $\tau(z)=z+c$. Let
	\begin{equation*}
	S_1\cup\cdots\cup S_l
	\end{equation*}
be the partition of $\{0,\ldots,n\}$ formed in such a way that $i$ and $j$ are in the same class $S_k$ if and only if $g_i/g_j\in \mathcal{P}^1_c$. If
	\begin{equation}\label{boreleq}
	g_0+\ldots + g_n=0,
	\end{equation}
then
	\begin{equation*}
	\sum_{i\in S_k} g_i =0
	\end{equation*}
for all $k\in\{1,\ldots,l\}$.
\end{theorem}

For the proof of Theorem~\ref{borelanalogue} we need two lemmas. The first one characterizes linear dependence of the coordinate functions of $g$ over the field $\mathcal{P}^1_c$.

\begin{lemma}\label{casoratilemma}
If the holomorphic curve $g=[g_0:\cdots:g_n]$ satisfies $\varsigma(g)<1$ and if $c\in\C$, then $C(g_0,\ldots,g_n)\equiv 0$ if and only if the entire functions $g_0,\ldots,g_n$ are linearly dependent over the field $\mathcal{P}^1_c$.
\end{lemma}

Since the periodic functions of $\mathcal{P}^1_c$ are constants with respect to the difference operator $\Delta_c f=f(z+c)-f(z)$, Lemma~\ref{casoratilemma} is a natural difference analogue of the fact that entire functions $f_0,\ldots,f_n$ are linearly dependent over $\C$ if and only if the Wronskian $W(f_0,\ldots,f_n)$ vanishes identically.

\medskip

\noindent\textit{Proof of Lemma~\ref{casoratilemma}: } Suppose first that $g_0,\ldots,g_n$ are linearly dependent over $\mathcal{P}^1_c$. Then there exist $A_0,\ldots,A_n\in\mathcal{P}^1_c$ such that
	$
	A_0g_0+\cdots+A_ng_n=0,
	$
and so
	\begin{equation}\label{system1}
	\left\{\begin{array}{lcl}
	A_0g_0+\cdots+A_ng_n &=&0\\
	A_0\overline{g}_0+\cdots+A_n \overline{g}_n&=&0\\
	&\vdots&\\
	A_0\overline{g}^{[n]}_0+\cdots+A_n \overline{g}^{[n]}_n&=&0.\\
	\end{array}\right.
	\end{equation}
The determinant of the coefficient matrix corresponding to the system \eqref{system1} is the Casoratian $C(g_0,\ldots,g_n)$. Since \eqref{system1} has a nontrivial solution, it follows that $C(g_0,\ldots,g_n)\equiv0$.

We apply induction on $n$ to prove the converse assertion. In the case when $n=1$ suppose that $C(g_0,g_1)\equiv0$, and consider the system of equations
	\begin{equation}\label{system2}
	\left\{\begin{array}{lcl}
	A_0g_0+A_1g_1 &=&0\\
	A_0\overline{g}_0+A_1 \overline{g}_1&=&0\\
	\end{array}\right.
	\end{equation}
which is equivalent to
	\begin{equation*}
	\left\{\begin{array}{lcl}
	A_0g_0+A_1g_1 &=&0\\
	A_1 C(g_0,g_1)&=&0.\\
	\end{array}\right.
	\end{equation*}
Since $C(g_0,g_1)\equiv0$, it follows that $A_0=g_1/g_0$ and $A_1=-1$ is a solution of \eqref{system2}. Moreover, since $\varsigma(g)<1$ by assumption, also $\varsigma(\tilde g)<1$ where $\tilde g =[g_0:g_1]$. Therefore, by \eqref{cartanT}, the usual Nevanlinna hyper-order of $A_0$ satisfies $\varsigma(A_0)=\varsigma(g_1/g_0)\leq \varsigma(\tilde g)\leq \varsigma(g)<1$. Since clearly $A_1\in\mathcal{P}^1_c$, all we need to do to complete the proof in the case $n=1$ is to show that $A_0$ is periodic with period $c$. By applying the difference operator $\Delta_c f=f(z+c)-f(z)$ to the first equation in \eqref{system2}, we have
	\begin{equation}\label{periodic1}
	A_0\Delta_c g_0+\overline g_0 \Delta_c A_0-\Delta_c g_1=0.
	\end{equation}
On the other hand, \eqref{system2} yields
	\begin{equation*}
	A_0\Delta_c g_0-\Delta_c g_1=0,
	\end{equation*}
which, combined with \eqref{periodic1}, implies that $\Delta_c A_0\equiv0$. We conclude that $A_0\in\mathcal{P}^1_c$.

Suppose now that $C(g_0,\ldots,g_j)\equiv0$ implies that $g_0,\ldots,g_j$ are linearly dependent over $\mathcal{P}^1_c$ for all $j\in\{1,\ldots,k-1\}$ where $k\leq n$, and assume that $C(g_0,\ldots,g_k)\equiv0$. Then the linear system
	\begin{equation}\label{system3}
	\left\{\begin{array}{lcl}
	A_0g_0+\cdots+A_kg_k &=&0\\
	A_0\overline{g}_0+\cdots+A_k \overline{g}_k&=&0\\
	&\vdots&\\
	A_0\overline{g}^{[k]}_0+\cdots+A_k \overline{g}^{[k]}_k&=&0\\
	\end{array}\right.
	\end{equation}
has at least one redundant equation and can be written as
	\begin{equation}\label{system4}
	\left\{\begin{array}{lcl}
	A_0g_0+\cdots+A_{k-1}g_{k-1} &=&g_k\\
	A_0\overline{g}_0+\cdots+A_{k-1} \overline{g}_{k-1}&=&\overline{g}_k\\
	&\vdots&\\
	A_0\overline{g}^{[k-1]}_0+\cdots+A_{k-1} \overline{g}^{[k-1]}_{k-1}&=&\overline{g}^{[k-1]}_k\\
	\end{array}\right.
	\end{equation}
where we have made the choice $A_k=-1$. If $C(g_0,\ldots,g_{k-1})\equiv0$, then $g_0,\ldots,g_{k-1}$ (and thus also $g_0,\ldots,g_{k}$) are linearly dependent over $\mathcal{P}^1_c$ by the induction assumption. If $C(g_0,\ldots,g_{k-1})\not\equiv0$, then by Cramer's rule for each $i=0,\ldots,k-1$ we have
	\begin{equation*}
	A_i=\frac{C(g_0,\ldots,g_k,\ldots,g_{k-1})}{C(g_0,\ldots,g_{k-1})},
	\end{equation*}
where $g_k$ occurs in the $i^\textrm{th}$ entry of the Casorati determinant in the numerator instead of $g_i$. By writing $A_i$ in the form
	\begin{equation*}
	A_i=\frac{g_i\overline{g}_i\cdots \overline{g}^{[k-1]}_i C(g_0/g_i,\ldots,g_k/g_i,\ldots,g_{k-1}/g_i)}			 {g_k\overline{g}_k\cdots \overline{g}^{[k-1]}_k C(g_0/g_k,\ldots,g_{k-1}/g_k)}
	\end{equation*}
it can be seen that
	\begin{equation*}
	 T(r,A_i)=O\left(\sum_{j=0}^{k}\sum_{l=0}^{k-1}\left(T\left(r,\frac{\overline{g}^{[l]}_j}{\overline{g}^{[l]}_i}\right)+T\left(r,\frac{\overline{g}^{[l]}_j}{\overline{g}^{[l]}_k}\right)\right)\right)
	\end{equation*}
for all $i=0,\ldots,k-1$. Now, since $T(r,f(z+c))=O(T(r+|c|,f))$ for all functions $f(z)$ meromorphic in the complex plane (see, e.g., \cite[pp.~66--67]{goldbergo:08}) it follows by the assumption $\varsigma(g)<1$ and \eqref{cartanT} that $\varsigma(A_i)<1$ for all $i=0,\ldots,k-1$.

We still need to prove that $A_i$ is periodic with period $c$ for all $i=0,\ldots,k-1$ ($A_k\equiv -1$ and so it is trivially periodic). By applying the operator $\Delta_c$ to all equations in the system \eqref{system4}, it follows that
	\begin{equation}\label{system5}
	\left\{\begin{array}{lcl}
	A_0\Delta_c g_0+\cdots+A_{k-1} \Delta_c g_{k-1} + \overline{g}_0 \Delta_c A_0 + \cdots + \overline{g}_{k-1} \Delta_c A_{k-1}  &=& \Delta_c g_k\\
	A_0\Delta_c \overline{g}_0+\cdots+A_{k-1} \Delta_c \overline{g}_{k-1} + \overline{\overline{g}}_0 \Delta_c A_0 + \cdots + \overline{\overline{g}}_{k-1} \Delta_c A_{k-1}  &=& \Delta_c \overline{g}_k\\
	&\vdots&\\
	A_0\Delta_c \overline{g}^{[k-1]}_0+\cdots+A_{k-1} \Delta_c \overline{g}^{[k-1]}_{k-1} + \overline{g}^{[k]}_0 \Delta_c A_0 + \cdots + \overline{g}^{[k]}_{k-1} \Delta_c A_{k-1}  &=& \Delta_c \overline{g}^{[k-1]}_k.\\
	\end{array}\right.
	\end{equation}
On the other hand, from \eqref{system3}, we obtain
	\begin{equation}\label{system6}
	\left\{\begin{array}{lcl}
	A_0\Delta_c g_0+\cdots+A_{k-1} \Delta_c g_{k-1}   &=& \Delta_c g_k\\
	A_0\Delta_c \overline{g}_0+\cdots+A_{k-1} \Delta_c \overline{g}_{k-1}   &=& \Delta_c \overline{g}_k\\
	&\vdots&\\
	A_0\Delta_c \overline{g}^{[k-1]}_0+\cdots+A_{k-1} \Delta_c \overline{g}^{[k-1]}_{k-1}   &=& \Delta_c \overline{g}^{[k-1]}_k.\\
	\end{array}\right.
	\end{equation}
By combining \eqref{system5} and \eqref{system6} we finally obtain
	\begin{equation}\label{system7}
	\left\{\begin{array}{lcl}
	 \overline{g}_0 \Delta_c A_0 + \cdots + \overline{g}_{k-1} \Delta_c A_{k-1}  &=& 0\\
	 \overline{\overline{g}}_0 \Delta_c A_0 + \cdots + \overline{\overline{g}}_{k-1} \Delta_c A_{k-1}  &=& 0\\
	&\vdots&\\
	 \overline{g}^{[k]}_0 \Delta_c A_0 + \cdots + \overline{g}^{[k]}_{k-1} \Delta_c A_{k-1}  &=& 0\\
	\end{array}\right.
	\end{equation}
which has only the trivial solution if $C(g_0,\ldots,g_{k-1})\not\equiv0$. Therefore $\Delta_c A_0\equiv \cdots\equiv \Delta_c A_{k-1}\equiv 0$, and so $A_i\in\mathcal{P}^1_c$ for all $i=0,\ldots,k-1$. \hfill$\Box$

\begin{lemma}\label{linindepP}
Let $c\in\C$, and let $g=[g_0:\cdots:g_n]$ be a holomorphic curve such that $\varsigma(g)<1$ and such that all zeros of $g_0,\ldots,g_n$ are forward invariant with respect to the translation $\tau(z)=z+c$. If $g_i/g_j\not\in \mathcal{P}^1_c$ for all $i,j\in\{0,\ldots,n\}$ such that $i\not=j$, then $g_0,\ldots,g_n$ are linearly independent over $\mathcal{P}^1_c$.
\end{lemma}

\begin{proof}
We will show that if $g_0,\ldots,g_n$ are linearly dependent over $\mathcal{P}^1_c$, then it follows that there exist $i,j\in\{0,\ldots,n\}$, $i\not=j$, such that $g_i/g_j\in \mathcal{P}^1_c$. Towards that end, suppose that $A_0,\ldots,A_n\in\mathcal{P}^1_c$ such that
	\begin{equation}\label{lindepAg}
	A_0g_0+\cdots + A_{n-1}g_{n-1}=A_n g_n
	\end{equation}
and not all $A_j$ are identically zero. Without loss of generality we may assume that none of the functions $A_j$ are identically zero. From the assumptions of the lemma and from the fact that $A_0,\ldots,A_n$ are periodic, it follows that there exists a meromorphic function $F(z)$ such that $FA_0g_0,\ldots,FA_ng_n$ are entire functions without common zeros and such that the preimages of all zeros of $FA_0g_0,\ldots,FA_ng_n$ are forward invariant with respect to the translation $\tau(z)$. Moreover, since $A_0,\ldots,A_n\in\mathcal{P}^1_c$, the function $F(z)$ satisfies
	\begin{equation}\label{NH1H}
	\limsup_{r\to\infty}\frac{\log^+\log^+ \left(N(r,F)+N(r,1/F)\right)}{\log r}<1
	\end{equation}
(but in general the hyper-order of $F$ might not be less than $1$). We define $G=[FA_0g_0:\cdots :FA_{n-1}g_{n-1}]$. Since $FA_0g_0,\ldots,FA_ng_n$ do not have any common zeros, it follows from \eqref{lindepAg} that $FA_0g_0,\ldots,FA_{n-1}g_{n-1}$ cannot have any common zeros either. Therefore the $T_G(r)$ is well defined. Furthermore
	\begin{equation}\label{TGtg}
	\begin{split}
	T_G(r)&=\int_0^{2\pi}\sup_{k\in \{0,\ldots,n-1\}} \log|FA_kg_k(re^{i\theta})|\frac{d\theta}{2\pi}+O(1)\\ &= \int_0^{2\pi}\sup_{k\in \{0,\ldots,n-1\}} (\log|g_k(re^{i\theta})|+\log|A_k(re^{i\theta})|)\frac{d\theta}{2\pi}\\&\quad+\int_0^{2\pi}\log|F(re^{i\theta})|\frac{d\theta}{2\pi}+O(1)\\ &\leq T_g(r)+\sum_{j=0}^{n-1} m(r,A_j) +\int_0^{2\pi}\log|F(re^{i\theta})|\frac{d\theta}{2\pi}+O(1).\\
  	\end{split}
	\end{equation}
Since Poisson-Jensen formula implies that
	\begin{equation}\label{PJformula}
	\int_0^{2\pi}\log|F(re^{i\theta})|\frac{d\theta}{2\pi}=N\left(r,\frac{1}{F}\right)-N(r,F)+O(1),
	\end{equation}
and since $A_0,\ldots,A_{n-1}\in\mathcal{P}^1_c$, it follows by combining \eqref{NH1H} and \eqref{TGtg} that $\varsigma(G)<1$. Suppose that $FA_0g_0,\ldots,FA_{n-1}g_{n-1}$ are linearly independent over $\mathcal{P}^1_c$. Then, $C(FA_0g_0,\ldots, FA_{n-1}g_{n-1})\not\equiv0$ by Lemma~\ref{casoratilemma}, and so Theorem~\ref{casorati} applied with $G$ and $FA_0g_0,\ldots,FA_ng_n$ yields
    \begin{equation}\label{thth}
    \begin{split}
    T_G(r) &\leq \sum_{j=0}^{n-1} N\left(r,1/FA_jg_j\right)
    + N\left(r,1/FA_ng_n\right)\\&\quad - N\left(r,1/C(FA_0g_0,\ldots, FA_{n-1}g_{n-1})\right)
    +o(T_G(r))+O(1)
    \end{split}
    \end{equation}
for all $r$ outside of an exceptional set of finite logarithmic
measure.

Since the preimages of all zeros of $FA_0g_0,\ldots,FA_ng_n$ are forward invariant with respect to $\tau(z)$, all zeros of
$FA_jg_j$, $j=0,\ldots,n-1$, are zeros of the Casorati determinant
$C(FA_0g_0,\ldots,FA_{n-1}g_{n-1})$ with the same or higher multiplicity.
Moreover, since $FA_0g_0,\ldots,FA_ng_n$ do not have any common zeros, it follows in particular that for each $z_0\in\C$ such that $FA_ng_n(z_0)=0$ with multiplicity $m_0$ there exist $k_0\in\{0,\ldots,n-1\}$ such that $FA_{k_0}g_{k_0}(z_0)\not=0$. Using \eqref{lindepAg} we may write
	\begin{equation*}
	\begin{split}
	&C(FA_0g_0,\ldots,FA_{n-1}g_{n-1})\\&\quad=C(FA_0g_0,
	\ldots,FA_{k_0-1}g_{k_0-1},FA_{n}g_{n},FA_{k_0+1}g_{k_0+1},
	\ldots,FA_{n-1}g_{n-1})
	\end{split}
	\end{equation*}	
which implies that $C(FA_0g_0,\ldots,FA_{n-1}g_{n-1})$ has a zero at $z_0$ with multiplicity $m_0$ at least. Also, at any common
zero of the functions $FA_{j_k}g_{j_k}$ with multiplicities $m_{j_k}$, $k=1,\ldots,\ell$, where $\{j_1,\ldots,j_\ell\}\subset\{0,\ldots,n-1\}$ and $\ell\leq n-2$, the Casorati determinant $C(FA_0g_0,\ldots,FA_{n-1}g_{n-1})$ has a zero of multiplicity $\geq \sum_{k=1}^\ell m_{j_k}$. Therefore,
    \begin{equation*}
    \sum_{j=0}^{n-1} N\left(r,1/FA_jg_j\right)
    + N\left(r,1/FA_ng_n\right)\leq
    N\left(r,1/C(FA_0g_0,\ldots,FA_{n-1}g_{n-1})\right),
    \end{equation*}
and so inequality \eqref{thth} yields $T_G(r)=O(1)$. But this is only possible when $G$ is a constant curve, which implies that $g_0,\ldots,g_{n-1}$ (and so also $FA_0g_0,\ldots,FA_{n-1}g_{n-1}$) are linearly dependent over $\mathcal{P}^1_c$. Therefore there exist $B_0,\ldots,B_{n-1}\in\mathcal{P}^1_c$ such that
	\begin{equation*}
	B_0g_0+\cdots + B_{n-2}g_{n-2}=B_{n-1} g_{n-1},
	\end{equation*}
where not all $B_j$ are identically zero. By continuing in this fashion it follows after at most $n-2$ iterations of the above reasoning that $g_i/g_j\in\mathcal{P}_c^1$ for some $i\not=j$.
\end{proof}

\noindent\textit{Proof of Theorem~\ref{borelanalogue}: } Using the fact that $g_i=A_{i,j_k}g_{j_k}$ for some $A_{i,j_k}\in\mathcal{P}_c^1$ whenever the indexes $i$ and $j_k$ are in the same class $S_{k}$, equation \eqref{boreleq} may be written as
	$$
	\sum_{k=0}^n g_k = \sum_{k=1}^l \sum_{i\in S_k} A_{i,j_k} g_{j_k}=\sum_{k=1}^l B_k g_{j_k}=0,
	$$
where $B_k= \sum_{i\in S_k} A_{i,j_k}$. By Lemma~\ref{linindepP} $B_k\equiv0$ for all $k=1,\ldots,l$, and so
	$$
	\sum_{i\in S_k} g_i = \sum_{i\in S_k} A_{i,j_k} g_{j_k} = B_k g_{j_k} =0
	$$	
for all $k=1,\ldots,l$. \hfill$\Box$

\section{Applications to the uniqueness of meromorphic
functions}\label{uniquenesssec}

Nevanlinna has shown that if two non-constant meromorphic functions $f$ and $g$
share five distinct values ignoring multiplicities (IM), then $f\equiv g$.
Similarly, if $f$ and $g$ share four values counting
multiplicities (CM), then there exists a M\"obius transformation $T$ such that $f\equiv T\circ g$.
These results are known as Nevanlinna's five and four values theorems,
respectively \cite{nevanlinna:29}. Gundersen has proved that the assertion of the four values theorem remains valid when
it is assumed that two values are shared IM and two are shared CM \cite{gundersen:83}. He has also given a counterexample
which demonstrates that in general this assumption cannot be further weakened to 4 IM \cite{gundersen:79}. The case where one value is shared CM and three values IM is still open.

A difference analogue of the five value theorem states that if instead of multiplicities we ignore those
values which have forward invariant preimages, then either $f\equiv g$ or both $f$ and $g$ are periodic
functions \cite{halburdk:06AASFM}. In this section we apply Lemma~\ref{linindepP} to show that the assumption 4 CM can be weakened to a difference analogue of the  4 IM assumption for meromorphic functions of hyper-order strictly less that one.

We denote by $\mathcal{S}(f)$ the set of all  meromorphic
functions $a$ such that $T(r,a)=o(T(r,f))$ where $r$ approaches
infinity outside of a set of finite logarithmic measure. Functions
in the set $\mathcal{S}(f)$ are called \textit{small compared to}
$f$, or \textit{slowly moving} with respect to $f$. Moreover, we say that two meromorphic functions $f$ and $g$ share a periodic function $a\in\mathcal{P}^1_c\setminus\{\infty\}$,
\textit{ignoring} $c$-\textit{separated pairs} (IcP), when for all $z\in\C$ exactly one of the following
assertions is valid for the ratio $r(z):=\{f(z)-a(z)\}/\{g(z)-a(z)\}$.
    \begin{itemize}
    \item[(i)] $r(z)$ is regular and does not vanish,
    \item[(ii)] $r(z)$ vanishes but $r(z+c)/r(z)$ is regular,
    \item[(iii)] $r(z)$ has a pole but $r(z)/r(z+c)$ is regular.
    \end{itemize}
We say also that two non-constant meromorphic functions $f$ and $g$ share
the constant function $a(z)\equiv \infty$ IcP, if their reciprocals $1/f$ and $1/g$ share the constant~$0$ IcP,
that is, if for all $z\in\C$ we have exactly one of the followings.
    \begin{itemize}
    \item[(i)] both $f(z)$ and $g(z)$ are regular,
    \item[(ii)] $f(z)$ is not regular but $f(z)g(z+c)/f(z+c)g(z)$ is regular,
    \item[(iii)] $g(z)$ is not regular but $f(z+c)g(z)/f(z)g(z+c)$ is regular.
    \end{itemize}
The following theorem is a difference analogue of the four value
theorem where 4~CM has been replaced by 4 IcP.

\begin{theorem}\label{uniqueness}
Let $c\in\C\setminus\{0\}$, and let $f$ and $g$ be meromorphic functions such that $\max\{\varsigma(f),\varsigma(g)\}<1$. If $f$ and $g$ share the distinct functions $a_1,a_2,a_3,a_4\in \mathcal{P}^1_c$ IcP, then
    \begin{equation}\label{mobius}
    f=\frac{Ag+B}{Cg+D},
    \end{equation}
where $A,B,C,D\in\mathcal{P}_c^1$.
\end{theorem}

Note that the functions $a_1,a_2,a_3,a_4$ need not be small
compared to $f$ or $g$. The following example shows that the
transformation \eqref{mobius} cannot be replaced by the identity
$f=g$.

\begin{example}
Denote by $\textrm{sn}(z,k)\equiv\textrm{sn}(z)$ the elliptic
function with the elliptic modulus $k\in(0,1)$ and the complete
elliptic integral $K$. The function $\textrm{sn}(z)$ is periodic
with the periods $4K$ and $2iK'$, and it attains the value zero at
points $2nK+2miK'$ and has its poles at $2nK+(2m+1)iK'$, where
$n,m\in\Z$. Therefore the meromorphic functions
    \begin{equation*}
    f(z)=\frac{\cos^2 (\pi z/K) +\sin^2 (\pi z/K)\, \sn(z)}{\cos
    (\pi z/K) +\sin (\pi z/K)\, \sn (z)}
    \end{equation*}
and
    \begin{equation*}
    g(z)=\frac{\cos^2 (\pi
    z/K)\, \sn (z) +  \sin^2 (\pi z/K) }{\cos (\pi z/K)\, \sn (z)
    +\sin (\pi z/K)}
    \end{equation*}
share four small periodic functions $\sin (\pi z/K)$, $\cos (\pi
z/K)$,
$$
\cos (\pi z/K) + \sin (\pi z/K) \quad \mbox{and} \quad
\frac{1}{\cos (\pi z/K) + \sin (\pi z/K)}
$$
ignoring $2K$-separated pairs.
\end{example}

\noindent\textit{Proof of Theorem \ref{uniqueness}: } By using a M\"obius
transformation over the field $\mathcal{P}_c^1$ (i.e. a so called quasi-M\"obius
transformation), if necessary, we may assume that $f$ and $g$
share $0,1,a,\infty$ IcP, where
$a\in\mathcal{P}_c^1\setminus\{0,1,\infty\}$. By \cite[Theorem 1]{miles:72}, there exist entire functions $\pi_j$ and $\kappa_j$ such that
    \begin{equation}\label{pi123}
    \frac{f}{g}=\frac{\pi_1}{\kappa_1},\quad
    \frac{f-1}{g-1}=\frac{\pi_2}{\kappa_2},\quad\frac{f-a}{g-a}=\frac{\pi_3}{\kappa_3},
    \end{equation}
where $\varsigma(\pi_j)<1$ and $\varsigma(\kappa_j)<1$ for $j=1,2,3$. Note that in general the functions $\pi_j$ and $\kappa_j$ may have common zeros for $j=1,2,3$. From \eqref{pi123} it follows that
    \begin{equation}\label{pieq}
    (1-a)\pi_1\kappa_2\kappa_3 +a\pi_2\kappa_1\kappa_3 -
    \pi_1\pi_2\kappa_3=\pi_3\kappa_1\kappa_2
    -a\pi_1\pi_3\kappa_2+(a-1)\pi_2\pi_3\kappa_1,
    \end{equation}
and so, by denoting
    \begin{equation}\label{gs}
    \begin{split}
    g_0&:=(1-a)\pi_1\kappa_2\kappa_3, \quad
    g_1:=a\pi_2\kappa_1\kappa_3,\quad
    g_2:= - \pi_1\pi_2\kappa_3, \\
    g_3&:=-\pi_3\kappa_1\kappa_2, \quad
    g_4:=a\pi_1\pi_3\kappa_2, \quad
    g_5:=(a-1)\pi_2\pi_3\kappa_1,
    \end{split}
    \end{equation}
it follows that
    \begin{equation}\label{g6}
    g_5=g_0+g_1+g_2+g_3+g_4.
    \end{equation}

Let $F(z)$ be a meromorphic function such that $Fg_0,\ldots,Fg_5$ are entire functions without common zeros. We will now show that the zeros of $Fg_0,\ldots,Fg_5$ are forward invariant with respect to the translation $z\mapsto z+c$. We will present a detailed proof of this fact only for the function $Fg_0=F(1-a)\pi_1\kappa_2\kappa_3$, since the functions $Fg_1,\ldots,Fg_5$ can be treated in an identical fashion. Suppose that $Fg_0$ has a zero of multiplicity $n$ at $z=z_0$. Since $Fg_0,\ldots,Fg_5$ are entire and they do not have any common zeros, it follows that the only possible zeros of $F$ are at the poles of $a$. Since $a$ is periodic, it follows that if $F(z_0)=0$ then $F(z_0+c)=0$ with the same multiplicity. Similarly, if $z_0$ is a zero of $a-1$, then there is a corresponding zero of $a-1$ at $z_0+c$, since $a$ is periodic. Assume that $z_0$ is a zero of $\pi_1$, and let $m\geq0$ be multiplicity of the (possible) zero of $\kappa_1$ at $z=z_0$. It follows by \eqref{gs} that $g_0,\ldots,g_5$ have a common zero of multiplicity at least $\min\{m,n\}$ arising from the common zero of $\pi_1$ and $\kappa_1$ at $z=z_0$. By the choice of $F$, it follows that $F$ has a pole at $z_0$, at least of order $\min\{m,n\}$, corresponding to the common zero of $\pi_1$ and $\kappa_1$. Combining the pole of $F$ with the zero of $\pi_1$, we obtain a zero of multiplicity $\max\{0,m-n\}$. On the other hand, $\pi_1$ has a zero of multiplicity $\max\{0,m-n\}$, at least, at $z=z_0+c$, since $f$ and $g$ share $0$ and $\infty$ IcP. The cases where $z_0$ is a zero of $\kappa_2$ or $\kappa_3$ can be treated similarly. Note that it may happen that some, or all, of the factors of $g_0=(1-a)\pi_1\kappa_2\kappa_3$ have a zero at $z=z_0$ simultaneously. This does not cause a difficulty, however, since these zeros will have a cumulative effect on order of the pole of $F$ at $z_0$.

Furthermore, similarly as in \eqref{NH1H}, \eqref{TGtg} and \eqref{PJformula}, it follows that the holomorphic curve $G=[Fg_0:\cdots:Fg_5]$ satisfies $\varsigma(G)<1$. Therefore, since $Fg_0,\ldots,Fg_5$ are linearly dependent by \eqref{g6}, it follows by Lemma~\ref{linindepP} that there exists a $\beta\in\mathcal{P}_c^1$ such that $g_k=\beta g_\ell$ for some $\ell,k\in\{0,\ldots,5\}$ where $\ell\not=k$. We may assume without a loss of
generality that $\ell >k$, and therefore by recalling the
definition \eqref{gs} it follows that $\beta$ is one of the
functions
    \begin{equation}\label{B}
    \begin{split}
    &\frac{g_0}{g_1}=\frac{(1-a)\pi_1\kappa_2}{a\pi_2\kappa_1},\quad
    \frac{g_0}{g_2}=\frac{(a-1)\kappa_2}{\pi_2}
    ,\quad \frac{g_0}{g_3}=\frac{(a-1)\pi_1\kappa_3}{\pi_3\kappa_1},\\
    &\frac{g_0}{g_4}=\frac{(1-a)\kappa_3}{a\pi_3},\quad \frac{g_0}{g_5}=-\frac{\pi_1\kappa_2\kappa_3}{\pi_2\pi_3\kappa_1},
    \quad\frac{g_1}{g_2}=-\frac{a\kappa_1}{\pi_1},
    \quad\frac{g_1}{g_3}=-\frac{a\pi_2\kappa_3}{\pi_3\kappa_2},\\
    & \frac{g_1}{g_4}=\frac{\pi_2\kappa_1\kappa_3}{\pi_1\pi_3\kappa_2},\quad \frac{g_1}{g_5}=\frac{a\kappa_3}{(a-1)\pi_3},
    \quad\frac{g_2}{g_3}=\frac{\pi_1\pi_2\kappa_3}{\pi_3\kappa_1\kappa_2},
    \quad \frac{g_2}{g_4}=-\frac{\pi_2\kappa_3}{a\pi_3\kappa_2}, \\ &\frac{g_2}{g_5}=\frac{
    \pi_1\kappa_3}{(1-a)\pi_3\kappa_1},\quad\frac{g_3}{g_4}=-\frac{\kappa_1}{a\pi_1},
    \quad \frac{g_3}{g_5}=\frac{\kappa_2}{(1-a)\pi_2},\quad
    \frac{g_4}{g_5}=\frac{a\pi_1\kappa_2}{(a-1)\pi_2\kappa_1}.
    \end{split}
    \end{equation}
Substituting \eqref{pi123} into \eqref{B} yields the desired
quasi-M\"obius transformation \eqref{mobius} in all cases
$\beta=g_k/g_\ell$ where $(k,\ell)$ is not one of the pairs
$(0,5)$, $(1,4)$ and $(2,3)$.

In order to deal with the remaining
cases, suppose first that $\beta=g_0/g_5$. Then, by \eqref{g6} it
follows that
    \begin{equation}\label{exceptioneq}
    (1-\beta)g_5=g_1+g_2+g_3+g_4.
    \end{equation}
By applying the early part of the proof to equation
\eqref{exceptioneq} instead of \eqref{g6}, it follows that
$g_{\tilde k}=\zeta g_{\tilde\ell}$ for some meromorphic function
$\zeta\in\mathcal{P}^1_c$ and distinct indexes $\tilde
k,\tilde\ell\in\{1,\ldots,5\}$. If $(\tilde k,\tilde\ell)$ is
neither $(1,4)$, nor $(2,3)$ then we are lead to one of the
quasi-M\"obius cases of \eqref{B}. Assume therefore that $(\tilde
k,\tilde\ell)=(1,4)$, which takes equation \eqref{exceptioneq} in
the form
    \begin{equation}\label{exceptioneq2}
    (1-\beta)g_5=g_2+g_3+(1+\zeta)g_4
    \end{equation}
Now there are two possibilities. If at least one of the factors
$(1-\beta)$ and $(1+\zeta)$ is non-zero, then there are at least
three identically non-zero terms in the equation
\eqref{exceptioneq2}, and the early part of the proof can be again
applied to deduce that $g_{\hat k}=\lambda g_{\hat\ell}$ for some
meromorphic function $\lambda\in\mathcal{P}^1_c$ and distinct indexes
$\hat k,\hat\ell\in\{2,\ldots,5\}$. The only possible
non-quasi-M\"obius case left in \eqref{B} is now $(\hat
k,\hat\ell)=(2,3)$, which, combined with equation
\eqref{exceptioneq2}, yields
    \begin{equation}\label{exceptioneq3}
    (1-\beta)g_5=(1+\lambda)g_3+(1+\zeta)g_4.
    \end{equation}
By performing the reduction operation one more time to equation
\eqref{exceptioneq3} yields one of the quasi-M\"obius cases of
\eqref{B}. We still need to consider the case where $(1-\beta)$ and
$(1+\zeta)$ both vanish. But then $g_0= g_5$ and $g_1=-g_4$,
which, together with \eqref{gs} and \eqref{B} imply that
    \begin{equation}\label{pi3kappa3}
    \frac{\pi_3}{\kappa_3}=\frac{\kappa_3}{\pi_3}.
    \end{equation}
By combining \eqref{pi123} and \eqref{pi3kappa3} it finally
follows that either $f=g$, or $f=-g+2a$. The cases $\beta=g_1/g_4$ and $\beta=g_2/g_3$ can be treated similarly. \hfill$\Box$

\section{Logarithmic difference estimate with applications to difference equations}\label{logdesec}

A difference analogue of the lemma on the logarithmic derivative for finite-order meromorphic functions was proved
independently by Halburd and Kor\-ho\-nen \cite[Lemma
2.3]{halburdk:06JMAA}, \cite[Theorem 2.1]{halburdk:06AASFM} and Chiang and Feng \cite[Theorem 2.4]{chiangf:08}. The following theorem is an extension of these results to the case of hyper-order less than one.

\begin{theorem}\label{logder}
Let $f$ be a non-constant meromorphic function, $\varepsilon>0$ and $c\in\C$. If
$f$ is of finite order, then there exists a set $E=E(f, \varepsilon)$ satisfying
    \begin{equation}\label{setE}
    \limsup_{r\to\infty} \displaystyle\frac{\int_{E\cap[1,r)}
    dt/t}{\log r} \leq \varepsilon,
    \end{equation}
i.e. of logarithmic density at most $\varepsilon$, such that
    \begin{equation}\label{finorder}
    m\left(r,\frac{f(z+c)}{f(z)}\right)=O\left(\frac{\log
    r}{r}\,T(r,f)\right)
    \end{equation}
for all $r$ outside the set $E$. If $\varsigma(f)=\varsigma<1$ and $\varepsilon>0$, then
    \begin{equation}\label{inforder}
    m\left(r,\frac{f(z+c)}{f(z)}\right)=o
    \left(\frac{T(r,f)}{r^{1-\varsigma-\varepsilon}}\right)
    \end{equation}
for all $r$ outside of a set of finite logarithmic measure.
\end{theorem}

Note that in case \eqref{inforder} the size of the exceptional set also depends on $\varepsilon$.
In the finite-order case we have aimed for the cleanest possible
statement with the expense of allowing a slightly larger
exceptional set. By following the proof of \eqref{inforder} in
Theorem \ref{logder}, it follows that for all finite-order
meromorphic functions $f$ the estimate
    \begin{equation*}
    m\left(r,\frac{f(z+c)}{f(z)}\right)\leq \frac{(\log r)^{3+\varepsilon}}{r}\,T(r,f)
    \end{equation*}
holds outside of an exceptional set of finite logarithmic measure.
Concerning the sharpness of these estimates, the gamma function
$\Gamma(z)$, for instance, satisfies
    \begin{equation*}
    m\left(r,\frac{\Gamma(z+1)}{\Gamma(z)}\right) = \log r \cong
    \frac{T(r,\Gamma)}{r}
    \end{equation*}
as $r$ approaches infinity. It is not clear, however, whether or
not the factor $\log r$ can be removed in \eqref{finorder}.

Concerning the estimate \eqref{inforder}, by defining $g(z):=\exp(2^z)$, it follows that
    \begin{equation*}
    m\left(r,\frac{g(z+1)}{g(z)}\right)=T(r,g)
    \end{equation*}
which shows that the growth condition $\varsigma(f)<1$ cannot be
essentially weakened. Also if $\sigma(f)<\infty$, then $\varsigma(f)=0$, and \eqref{inforder} reduces precisely into \cite[Theorem 2.1]{halburdk:06AASFM}.

We will now briefly discuss some applications of Theorem~\ref{logder} in the theory of difference equations. Yanagihara \cite{yanagihara:80} has shown that if $w$ is a non-rational meromorphic solution of the difference equation
    \begin{equation}\label{1order}
    w(z+1) = R(z,w),
    \end{equation}
where $R(z,w)$ is rational in $w$ having rational coefficients,
then, for any $\varepsilon>0$ and $r$ sufficiently large,
$T(r,w)\geq [(1-\varepsilon)\deg_w(R)]^r$. Therefore, if
\eqref{1order} has at least one non-rational meromorphic solution
$w$ of hyper-order strictly less than one, it follows that
$\deg_w(R)=1$, that is, equation \eqref{1order} reduces to the
difference Riccati equation. This result is sharp in the sense
that the difference equation $w(z+1)=w(z)^2$ is satisfied by
$w(z)=\exp(2^z)$, which has hyper-order exactly one. In the
second-order case, it follows by the proof of \cite[Theorem
3]{ablowitzhh:00} that if
    \begin{equation}\label{start}
    w(z+1) + w(z-1) = R(z,w)
    \end{equation}
has a meromorphic solution of hyper-order less than one, then
$\deg_w(R)\leq2$. In \cite{halburdk:07PLMS} it was shown that if
\eqref{start} has at least one admissible finite-order meromorphic
solution $w$, then either $w$ satisfies a difference Riccati
equation, or equation \eqref{start} can be transformed into one in
a list of equations consisting of difference Painlev\'e equations
and linear equations. Recall that a meromorphic solution $w$ of a difference (or differential) equation is called \textit{admissible} if all coefficients of the equation are in $\mathcal{S}(w)$. Following \cite{halburdk:07PLMS}, and by
applying Theorem \ref{logder} instead of \cite[Lemma
2.3]{halburdk:06JMAA}, the following version of \cite[Theorem
1.1]{halburdk:07PLMS} is obtained.

\begin{theorem}\label{pthm}
If the equation \eqref{start} where $R(z,w)$ is rational in $w$
and meromorphic in $z$, has an admissible meromorphic solution $w$
such that $\varsigma(w)<1$, then either $w$ satisfies a difference
Riccati equation, or equation \eqref{start} can be transformed by
a linear change in $w$ to one of the following equations:
    \begin{eqnarray*}
    \wu+ w + \wdn &=& \frac{\pi_1 z + \pi_2}{w} + \kappa_1\nonumber   \\
    \wu- w + \wdn &=& \frac{\pi_1 z + \pi_2}{w} + (-1)^z\kappa_1\nonumber \\
    \wu+\wdn &=& \frac{\pi_1 z + \pi_3}{w} + \pi_2  \nonumber \\
    \wu+\wdn &=& \frac{\pi_1 z + \kappa_1}{w} + \frac{\pi_2}{w^2}\nonumber   \\
    \wu+\wdn &=&\frac{(\pi_1 z+\kappa_1)w+\pi_2}{(-1)^{-z}-w^2}  \nonumber \\
    \wu+\wdn &=&\frac{(\pi_1 z+\kappa_1)w+\pi_2}{1-w^2} \nonumber \\
     \wu w+w\wdn &=& p \nonumber   \\
    \wu+\wdn &=& p\,w+q \nonumber\\
    \end{eqnarray*}
where $\overline{w}\equiv w(z+1)$, $\underline{w}\equiv w(z-1)$ and  $\pi_k,\kappa_k$ are periodic functions with period~$k$.
\end{theorem}

\begin{remark}
In order to prove Theorem~\ref{pthm} one needs to extend the difference analogues of Clunie's and Mohon'ko's theorems used in \cite{halburdk:07PLMS} to meromorphic solutions of difference equations of hyper-order strictly less than one. This can be done by combining Theorem~\ref{logder} with the proofs of \cite[Theorem 3.1--3.2]{halburdk:06JMAA}. Similarly one can extend the generalization of \cite[Theorem 3.1]{halburdk:06JMAA} obtained in \cite{lainey:07} to meromorphic solutions $f$ such that $\varsigma(f)<1$.
\end{remark}

\section{Counterparts in $q$-shifts}\label{section_q}

In this section we state a $q$-difference analogue of Theorem
\ref{hypers}. Similarly to the Casorati determinant, we define the
$q$-\textit{Casorati determinant} of entire functions
$g_0,\ldots,g_n$ by
    \begin{equation*}
    \hat{C}(g_0(z),\ldots,g_n(z))=\left|\begin{array}{ccccc}
      g_0(z) & g_1(z) & \cdots & g_n(z) \\
      g_0(qz) & q_1(qz) & \cdots & g_n(qz) \\
      \vdots & \vdots & \ddots & \vdots \\
      g_0(q^{n}z) & g_1(q^{n}z) &  \cdots & g_n(q^{n}z) \\
    \end{array}\right|.
    \end{equation*}
If $q\in\C\setminus\{0,1\}$, then the $q$-Casorati determinant vanishes identically on~$\C$
if and only if the functions $g_0,\ldots,g_n$ are linearly
dependent over the field of functions $\phi(z)$ satisfying
$\phi(qz)\equiv \phi(z)$. However, if $|q|\not=1$, then the intersection
of this field with the field of meromorphic functions
consists only of constant functions, and we are therefore
restricted to study hyperplanes over $\C$ in this context.

\begin{theorem}\label{hypersq}
Let $f:\C\to\P^n$ be a holomorphic curve such that
$\sigma(f)=0$, let $q\in\C\setminus\{0,1\}$ and let $p\in\{1,\ldots,n+1\}$. If $n+p$ hyperplanes in general position have forward invariant preimages
under $f$ with respect to the rescaling $\tau(z)= qz$, then the image of $f$ is contained in a projective linear subspace of dimension $\leq[n/p]$.
\end{theorem}

Theorem \ref{hypersq} can be proved by finding $q$-analogues of Theorems~\ref{casorati} and \ref{borelanalogue} and adapting of the proof of Theorem \ref{hypers} suitably, where the
$q$-difference analogue of the lemma on the logarithmic
derivatives from \cite{barnetthkm:07} is used in the place of Theorem~\ref{logder}. We omit further details of
the proof. The following corollary is an immediate consequence of Theorem~\ref{hypersq}.

\begin{corollary}\label{hypersqcor}
Let $f:\C\to\P^n$ be a holomorphic curve such that
$\sigma(f)=0$, and let $q\in\C\setminus\{0,1\}$. If $2n+1$
hyperplanes in general position have forward invariant preimages
under $f$ with respect to the
rescaling $\tau(z)= qz$, then $f$ is a constant.
\end{corollary}

The order condition $\sigma(f)=0$ in Theorem~\ref{hypersq}
and Corollary~\ref{hypersqcor} cannot be simply dropped in the following sense.

\begin{example}\label{qexample}
Five hyperplanes in $\mathbb{P}^2$ given by equations
$h_1(w):=w_0=0, h_2(w):=w_1=0, h_3(w):=w_2=0, h_4(w):=w_0+w_1+w_2=0, h_5(w):=w_0+\omega w_1+ \omega^2 w_2=0$
in general position have forward invariant preimages under the non-constant curve
$f=[1:\omega:e^z]:\mathbb{C}\to\mathbb{P}^2$ with $\omega= e^{2\pi i/3}$
with respect to the rescaling $\tau(z)=4z$.
In fact, $h_1(f)\equiv 1, h_2(f)\equiv \omega, h_3(f)=e^z$ are zero-free and the zeros of
$h_4(f)=e^z-\omega^2, h_5(f)=\omega^2(e^z-\omega^2)$ are forward invariant,
while $\sigma(f)=1$.
\end{example}

Let $f$ be a holomorphic curve such that $\sigma(f)=0$ given by $f(z):=[1:\omega:\Pi(z)+\omega^2]$ with the infinite product $\Pi(z)=\prod_{j=0}^{\infty}(1-z/q^j)$ when $|q|>1$. This $\Pi(z)$ satisfies $\Pi(qz)=(1-qz)\Pi(z)$ and therefore the zeros are forward invariant with respect to $\tau(z)=qz$, while we have $h_1(f)\equiv 1, h_2(f)\equiv \omega$ again and also $h_4(f)=\Pi(z), h_5(f)=\omega^2\Pi(z)$ now. Of course, the zeros of $h_3(f)=\Pi(z)-\omega$ cannot be kept forward invariant anymore with the rescaling.
This difference with Example~\ref{qexample} appears to stem from the fact that any non-constant entire function does not permit finite Picard exceptional values when the order of growth is less than one. This seems to indicate that Theorem~\ref{hypersq}
and Corollary~\ref{hypersqcor} should remain true when $0<\sigma(f)<1$, but at the moment we have no proof of this. Confirming this conjecture would require a different approach to the one used here, since it has been shown by examples that the $q$-difference analogue of the lemma on the logarithmic derivative obtained in \cite{barnetthkm:07} cannot be extended to meromorphic functions of non-zero order. Moreover, if $g$ is a transcendental entire function whose zeros are forward invariant with respect to a rescaling $r(z)=qz$, $q \not\in\{0,1\}$, then the function $\phi(z):=g(e^z)$ is entire and has the zeros that are forward invariant with respect to the shift $s(z)=z+c$ with $c=\log q \neq 0$. By an estimate due to J. Clunie in \cite{clunie:70},
it follows that
$$
\log \max_{|w|=e^r} |g(w)|+O(1)\geq \log \max_{|z|=r}|\phi(z)| \geq \log \max_{|w|=ae^{br}}|g(w)|+O(1)
$$
holds for some positive constants $a,b$ with $b<1$.
Therefore, we see that  $\varsigma(\phi)\leq 1$ if the order $\sigma(g)$ of $g$ is finite, while $\sigma(g)=0$ if $\varsigma(\phi)<1$.
This delicate growth balance between the functions $g$ and $\phi$ must be taken into account in any attempt trying to demonstrate the conjecture.

\medskip

\noindent \textbf{Note:} Any automorphism of $\C$ has the form
$\tau(z)=qz+r$, $q\not=0$, which is a composition of the shift
$z+r/q$ and rescaling $qz$.

\section{Sharpness of Theorems \ref{hypers} and
\ref{hypersq}}\label{examplesec}

Using similar methods to Green~\cite{green:72}, we see that the
dimension~$[n/p]$ in Theorem~\ref{hypers} is the sharpest possible bound and is always attained for any given $n+p$ hyperplanes in general position. In
fact, we only need to replace the choice of the exponential
functions $\exp (g_m)$ with holomorphic mappings
$g_m:\mathbb{C}^{[n/p]}\to\mathbb{P}^n$ $(m=1, \ldots, [n/p]+1)$
by the entire functions $1/\Gamma(z/c+\omega^m)$ on $\mathbb{C}$
with the ordinary gamma function $\Gamma(z)$ and the primitive
$([n/p]+1)$th root of unity $\omega$, for instance. Let $\mathcal{P}_c$ denote the field of periodic meromorphic functions with period $c$. Following
Green's argument, we obtain the holomorphic curve
$f:\mathbb{C}\to\mathbb{P}^n$ of order of growth one, whose image
is nondegenerately included in an $[n/p]$ dimensional linear
subspace of $\mathbb{P}^n$ and under which the $n+p$ hyperplanes
over $\mathcal{P}_{f}=\{\pi\in\mathcal{P}_c:T(r,\pi)=o(r\log r)\}$
have forward invariant preimages with respect to the
transformation $\tau(z)=z+c$. A similar argument will be used to demonstrate the sharpness of Theorem~\ref{hypersq}.

Let $f:\mathbb{C}\to\mathbb{P}^n$ be a linearly
non-degenerate holomorphic curve and $\{H_j\}_{j=1}^q$ be a family
of hyperplanes $H_j\subset \mathbb{P}^n$ in general
position. Instead of considering each hyperplane~$H_j$ itself
which is defined by
$$
\hat{H}_j(w):=\sum_{k=0}^n h_{jk}w_k =0 \quad (1\leq j\leq q)\,,
$$
we will mainly observe its representing
vector~$\mbox{\boldmath$h$}_j=(h_{j0},\, \cdots \, ,\, h_{jn}) \in
\mathbb{C}^{n+1}$. Here we recall that $w=[w_0:\cdots :w_n]$ is a
homogeneous coordinate system of $\mathbb{P}^n$. Then
it is convenient to use a symbol~$(\bullet,\, \bullet)$ to denote
a kind of `inner product' in $\mathbb{C}^{n+1}$ given by
$\bigl(\mbox{\boldmath$h$}_j,\,w\bigr):=\sum_{k=0}^n h_{jk}
w_k=\hat{H}_j(w).$ Let $f:=[f_0\, : \, \cdots \,:\, f_n]$ be a
reduced representation of the
curve~$f:\mathbb{C}\to\mathbb{P}^n$. Then
$\bigl(\mbox{\boldmath$h$}_j,\,f(z)\bigr):=\sum_{k=0}^n h_{jk}
f_k(z)=\hat{H}_j\bigl(f(z)\bigr)$ is an entire function
on~$\mathbb{C}$ for every~$j$. By
$\{\mbox{\boldmath$e$}_{k}\}_{k=0}^{n}$ we denote the standard
basis of~$\mathbb{C}^{n+1}$ throughout in this note, so that we
have $(\mbox{\boldmath$e$}_{k},\,f)=f_{k}$ $(0\leq k\leq n)$.

Let $m$ be a prime number and $\varepsilon$ be a primitive $m$th
root of unity. We take the set of $2m$ vectors ${\mathcal
H}:=\{\mbox{\boldmath$h$}_j : 1\leq j \leq 2m\}\subset
\mathbb{C}^{\, m}$ with
\[
\mbox{\boldmath$h$}_j= \left\{ \
\begin{array}{ll}
\mbox{\boldmath$e$}_{j-1} & (1\leq j\leq m)\,, \\[2ex]
\mbox{\boldmath$v$}_{j-m} & (m+1\leq j \leq 2m)\,,
\end{array}
\right.
\]
where
\[
\left(
\begin{array}{c}
\mbox{\boldmath$v$}_{1}     \\[1ex]
\mbox{\boldmath$v$}_{2}     \\[1ex]
\vdots  \\[1ex]
\mbox{\boldmath$v$}_{j} \\[1ex]
\vdots  \\[1ex]
\mbox{\boldmath$v$}_{m-1}   \\[1ex]
\mbox{\boldmath$v$}_{m} \\[1ex]
\end{array} \right)
= \left( \begin{array}{cccccc}
1 & 1 & \cdots & 1 &\cdots& 1\\[1ex]
1 & \varepsilon & \cdots & \varepsilon^k &\cdots&\varepsilon^{m-1}\\[1ex]
\vdots&\vdots&\ddots&\vdots&\ddots&\vdots\\[1ex]
1&\varepsilon^{j-1} &\cdots&\varepsilon^{(j-1) k}&\cdots&\varepsilon^{(j-1) (m-1)}\\[1ex]
\vdots&\vdots&\ddots&\vdots&\ddots&\vdots\\[1ex]
1&\varepsilon^{m-2}&\cdots&\varepsilon^{(m-2)k}&\cdots&\varepsilon^{(m-2)(m-1)}\\[1ex]
1&\varepsilon^{m-1}&\cdots&\varepsilon^{(m-1)k}&\cdots&\varepsilon^{(m-1)(m-1)}
\end{array} \right)\,,
\]
which is a regular $m$-matrix $V_m=(\varepsilon^{\ell k})$ $(0\leq
\ell,\, k\leq m-1)$, in fact, a Vandermonde matrix.

Then we see that any $m$ of the $2m$ vectors
$\mbox{\boldmath$h$}_{j}$ in ${\mathcal H}$ are linearly
independent over $\mathbb{C}$ in $\mathbb{C}^{\, m}$, so that
those $2m$ vectors give the family of $2m$ hyperplanes in
$\mathbb{P}^{\, m-1}(\mathbb{C})$ which are actually located in
general position. In order to confirm this matter, we only need to
know that every minor determinant of our Vandermonde matrix~$V_m$
does not vanish. As a matter of fact, it is only the reason why we
have chosen $m$ as a prime number that we can apply the following
lemma for the purpose:

\begin{lemma}[\cite{evansi:76}]
Let $m$ be a prime and let $\varepsilon$ be a primitive $m$th root
of unity in some field of characteristic zero. Suppose $a_1,\,
\ldots ,\, a_{\mu}\in\mathbb{Z}$ are pairwise incongruent
$(\!\!\!\mod m)$ and suppose the same for $b_1,\, \ldots ,\,
b_{\mu}\in\mathbb{Z}$. Then the determinant of the matrix
$\big(\varepsilon^{a_ib_j}\big)$ does not vanish.
\end{lemma}

\noindent One sees that this is not the case unless $m$ is prime:
for example, when $m=4$, $\varepsilon=\sqrt{-1}$, $a_1=b_1=1$ and
$a_2=b_2=3$, then the determinant does vanish.

Now we assume concretely $m=11$ so that $n=10$. When $p=3$, we
consider four entire functions $\phi_j(z)$ to be determined
concretely later and define the linearly non-degenerate
holomorphic curve $f:\mathbb{C}\to \mathbb{P}^{10}(\mathbb{C})$ by
$$
f(z):=\bigl[c_1\phi_1:c_2\phi_1:c_3\phi_1:
c_1\phi_2:c_2\phi_2:c_3\phi_2:c_1\phi_3:c_2\phi_3:c_3\phi_3:c_1\phi_4:c_2\phi_4\bigr]
$$
with some non-zero constants $c_i$ $(i=1,2,3)$ and take the set of
$22$ vectors ${\mathcal H}:=\{\mbox{\boldmath$h$}_j : 1\leq j \leq
22\}\subset \mathbb{C}^{11}$ as above:
\[
\mbox{\boldmath$h$}_j= \left\{ \
\begin{array}{ll}
\mbox{\boldmath$e$}_{j-1} & (1\leq j\leq 11)\,, \\[2ex]
\mbox{\boldmath$v$}_{j-11} & (12\leq j \leq 22)\,.
\end{array}
\right.
\]
Choosing $c_1=\varepsilon$, $c_2=-(\varepsilon +1)$
and $c_3=1$, we have
$$
c_1+ c_2+ c_3=0 \quad \text{and} \quad
c_1+\varepsilon c_2 +\varepsilon^2 c_3=0
$$
as well as $c_1+ c_2=-1$ and $c_1+\varepsilon c_2=-\varepsilon^2$.
Then $n+p=13$ vectors
$\mbox{\boldmath$h$}_{j}$ $(1\leq j\leq 13)$ give
$$
\bigl(\mbox{\boldmath$h$}_{j}, f(z)\bigr)=d_j\phi_{k_j}(z) \quad
(1\leq j\leq 11)
$$
with $d_j=c_i$ $(j\equiv i \mod 3)$ and $k_j=[(j-1)/3]+1$ $(1\leq j\leq 11)$
and also
$$
\bigl(\mbox{\boldmath$h$}_{12}, f(z)\bigr)=-\phi_{4}(z)\,,
\quad
\bigl(\mbox{\boldmath$h$}_{13}, f(z)\bigr)=-\varepsilon^2 \phi_{4}(z)\,,
$$
both of which follow from the choice of the three constants $c_i$ $(i=1,2,3)$.
On the other hand, it follows from the definition that
$f(\mathbb{C})$ is in the linear subspace of $\mathbb{P}^{10}(\mathbb{C})$
with dimension $3=[10/3]$.

In the same way, given any $p$ $(1\leq p\leq 10)$, we can obtain a
desired curve $f(z):\mathbb{C}\to\mathbb{P}^{10}(\mathbb{C})$
$$
f(z):=\bigl[\, \overbrace{\underbrace{c_1\phi_1: \, \cdots \,
:c_{p}\phi_1}_{p}: \cdots : \underbrace{c_1\phi_{s-1}: \, \cdots
\, :c_{p}\phi_{s-1}}_{p}}^{p\times [10/p]}: \underbrace{c_1\phi_s:
\, \cdots \, :c_{11-p[10/p]}\phi_s}_{11-p[10/p]} \, \bigr]
$$
for $s=[10/p]+1$ entire functions $\phi_k$ $(1\leq k\leq s)$ and
non-zero constants $c_i$ $(1\leq i\leq p)$ satisfying the
simultaneous linear equations
$$
\sum_{i=1}^p \varepsilon^{\ell (i-1)}c_i=0 \quad (0\leq \ell \leq p-1),
$$
together with $p+10$ hyperplanes defined by
$\mbox{\boldmath$h$}_{j}$ $(1\leq j\leq p+10)$ which satisfy
$$
\bigl(\mbox{\boldmath$h$}_{j}, f(z)\bigr)=d_j\phi_{k_j}(z) \quad
(1\leq j\leq 11)
$$
with $d_j=c_i$ $(j\equiv i \mod p)$ and $k_j=[(j-1)/p]+1$ $(1\leq
j\leq 11)$ and also
$$
\bigl(\mbox{\boldmath$h$}_{j}, f(z)\bigr)= d_j\phi_{s}(z) \quad
(12\leq j\leq p+10)
$$
with $d_j=-\sum_{i=1}^{11-p[10/p]}
\varepsilon^{(j-12)(i-1)}c_i(\neq 0)$ $(12\leq j \leq p+10)$,
since $11-p[10/p]$ does not coincide with $p$.

For each prime number $m$, we can similarly construct a
corresponding holomorphic curve $f(z)$ following the idea of Green.
When $n+1$ is not a prime, the choice of suitable hyperplanes
would be a little complicated.

\begin{remark}
According to the value of $p$, the number $11-p[10/p]$ varies as
follows:
\[
\begin{tabular}{|c||c|c|c|c|c|c|c|c|c|c|} \hline
$p$ & 1 & 2 & 3 & 4 & 5 & 6 & 7 & 8 & 9 & 10 \\ \hline
$11-p[10/p]$ & 1 & 1 & 2 & 3 & 1 & 5 & 4 & 3 & 2 & 1 \\ \hline
\end{tabular}
\]
Therefore, if $p$ is not any divisor of $n=10$ so that
$[10/p]\neq 10/p$, then we can obtain a curve
$g:\mathbb{C}\to\mathbb{P}^{9}(\mathbb{C})$ with our desired
properties by projecting the curve~$f$ into
$\mathbb{P}^{10}(\mathbb{C})$ and $p+10$ vectors
$\mbox{\boldmath$h$}_{j}\in\mathbb{C}^{11}$ given above into the
subspace
$$
\bigl\{[w_0:w_1:\cdots:w_9:1]\, |\,
[w_0:w_1:\cdots:w_9]\in\mathbb{P}^{9}(\mathbb{C})\bigr\}
$$
and $\mathbb{C}^{10}_{11}:=\mathbb{C}^{10}\times\{0\}$,
respectively. Concretely consider the case when $p=3$. Then we
give $g:\mathbb{C}\to \mathbb{P}^{9}(\mathbb{C})$ by
$$
g(z):=\bigl[c_1\phi_1:c_2\phi_1:c_3\phi_1:
c_1\phi_2:c_2\phi_2:c_3\phi_2:c_1\phi_3:c_2\phi_3:c_3\phi_3:c_1\phi_4\bigr]
$$
with constants $c_1=\varepsilon$, $c_2=-(\varepsilon+1)$ and $c_3=1$
as well as the following $12$ vectors in
$\mathbb{C}^{10}$:
\[
\mbox{\boldmath$\hat{h}$}_j= \left\{ \
\begin{array}{ll}
\mbox{\boldmath$\hat{e}$}_{j-1} &:=\mbox{\boldmath$e$}_{j-1}\cap \mathbb{C}^{10} \ (1\leq j\leq 10)\,, \\[2ex]
\mbox{\boldmath$\hat{v}$}_{j-11} &:=(1,1,\cdots , 1) \ (j=12) \,, \\[2ex]
\mbox{\boldmath$\hat{v}$}_{j-11} &:=(1,\varepsilon,\cdots , \varepsilon^{9}) \ (j=13)\,.
\end{array}
\right.
\]
Then we have
$$
\bigl(\mbox{\boldmath$\hat{h}$}_{j}, g(z)\bigr)=d_j\phi_{k_j}(z)
\quad (1\leq j\leq 10)
$$
with $d_j=c_i$ $(j\equiv i \mod 3)$ and $k_j=[(j-1)/3]+1$ $(1\leq
j\leq 10)$ and also
$$
\bigl(\mbox{\boldmath$\hat{h}$}_{12},
g(z)\bigr)=\varepsilon \phi_{4}(z)\,, \quad
\bigl(\mbox{\boldmath$\hat{h}$}_{13},
g(z)\bigr)= \varepsilon^{-1} \phi_{4}(z)\,.
$$
On the other hand, it follows from the definition that
$g(\mathbb{C})$ is in the linear subspace of
$\mathbb{P}^{9}(\mathbb{C})$ with dimension $3=[9/3]$.

In general we give
$$
g(z):=\bigl[\, \overbrace{\underbrace{c_1\phi_1: \, \cdots \,
:c_{p}\phi_1}_{p}: \cdots : \underbrace{c_1\phi_{s-1}: \, \cdots
\, :c_{p}\phi_{s-1}}_{p}}^{p\times [10/p]}: \underbrace{c_1\phi_s:
\, \cdots \, :c_{10-p[10/p]}\phi_s}_{10-p[10/p]} \, \bigr]
$$
for $s=[10/p]+1$ entire functions $\phi_k$ $(1\leq k\leq s)$ and
non-zero constants $c_i$ $(1\leq i\leq p)$ satisfying the
simultaneous linear equations
$$
\sum_{i=1}^p \varepsilon^{\ell (i-1)}c_i=0 \quad (0\leq \ell \leq p-1),
$$
together with $p+9$ hyperplanes defined respectively by the vector
$\mbox{\boldmath$\hat{h}$}_{j}$ which is the projection of
$\mbox{\boldmath$h$}_{j}$ on $\mathbb{C}^{10}_{11}$ for each $j
(\neq 11)$ with $1\leq j\leq p+10$. They still satisfy
$$
\bigl(\mbox{\boldmath$\hat{h}$}_{j}, g(z)\bigr)=d_j\phi_{k_j}(z)
\quad (1\leq j\leq 10)
$$
with $d_j=c_i$ $(j\equiv i \mod p)$ and $k_j=[(j-1)/p]+1$ $(1\leq
j\leq 10)$ and also
$$
\bigl(\mbox{\boldmath$\hat{h}$}_{j}, g(z)\bigr)= d_j\phi_{s}(z)
\quad (12\leq j\leq p+10)
$$
with $d_j=-\sum_{i=1}^{10-p[10/p]}
\varepsilon^{(j-12)(i-1)}c_i(\neq 0)$ $(12\leq j \leq p+10)$,
since $p>10-p[10/p]>0$.
\end{remark}

For any $q\in\C$ such that $|q|\in(0,1)$ the $q$-Gamma function
$\Gamma_q(x)$ is defined by
    $$
    \Gamma_q(x):=\frac{(q;q)_\infty}{(q^x;q)_\infty}(1-q)^{1-x},
    $$
where $(a;q)_\infty:=\prod_{k=0}^\infty(1-aq^k)$
\cite{andrewsar:99}. By defining
    $$
    \gamma_q(z):=(1-q)^{x-1}\Gamma_q(x), \qquad z=q^x,
    $$
and $\gamma_q(0):=(q;q)_\infty$, it follows that $\gamma_q(z)$ is
a zero-order meromorphic function with no zeros and having its
poles exactly at the points $\{q^{-k}\}_{k=0}^\infty$. Therefore
the preimages of the poles of $\gamma_q(z)$ are forward invariant
with respect to the rescaling $\tau(z)=qz$. To show the sharpness
of Theorems \ref{hypers} and \ref{hypersq}, we may take the
functions $\phi_j(z)$ by
    $$
    \phi_j(z)=\frac{1}{\Gamma \bigl(z+(j-1)/2\bigr)} \quad \text{and}
    \quad \phi_j(z) =\frac{1}{\gamma_q(q^{(j-1)/2}z)},
    $$
respectively.

\section{The proof of Theorem \ref{logder}}\label{section_lemma}

\begin{lemma}\label{ineqlemma}
Let $a\in\C$, $c\in\C$ and $\delta\in(0,1)$. Then
    \begin{equation}\label{assertn}
    \int_0^{2\pi}\log^{+}\left|1+\frac{c}{re^{i\theta}-a}\right|\,\frac{d\theta}{2\pi} \leq
    \frac{1}{\delta}\log^{+}\left(1+\frac{|c|^\delta}{(1-\delta)}\frac{1}{r^\delta}\right)
    \end{equation}
for all $r>0.$
\end{lemma}

\begin{proof} By Jensen's inequality \cite[p.~48]{cherryy:01}, it follows that
    \begin{equation}\label{eq1}
    \begin{split}
    \int_0^{2\pi}\log^{+}\left|1+\frac{c}{re^{i\theta}-a}\right|\,\frac{d\theta}{2\pi} &\leq
    \frac{1}{\delta}\int_0^{2\pi} \log^{+}\left(1+\left|\frac{c}{re^{i\theta}-
    a}\right|^\delta\right) \,\frac{d\theta}{2\pi} \\
    &\leq
    \frac{1}{\delta}\log^{+}\int_0^{2\pi} \left(1+\left|\frac{c}{re^{i\theta}-
    a}\right|^\delta\right) \,\frac{d\theta}{2\pi}
    \end{split}
    \end{equation}
for all $r>0$. Since $|re^{i\theta}-|a|| \geq
r\theta\frac{2}{\pi}$ for all $0\leq\theta\leq\frac{\pi}{2}$ and
any $a\in\C$ (see, e.g., \cite[p.~118]{goldbergo:08}), we have
    \begin{equation}\label{aux1}
    \int_0^{2\pi}\frac{d\theta}{|re^{i\theta}-a|^\delta} \leq
    4\int_0^{\frac{\pi}{2}}\frac{d\theta}{|re^{i\theta}-|a||^\delta} \leq \frac{2\pi}{1-
    \delta}\frac{1}{r^\delta}
    \end{equation}
whenever $\delta\in(0,1)$. Inequality \eqref{assertn} follows by
combining \eqref{eq1} and \eqref{aux1}.
\end{proof}

The following lemma is an improved version of the inequality
obtained in \cite[Lemma 2.3]{halburdk:06JMAA} (see also
\cite[Theorem 2.4]{chiangf:08}). Its method of proof is based on a combination of the techniques used in the proofs of \cite[Lemma 3]{hinkkanen:92} and
\cite[Lemma 2.3]{halburdk:06JMAA}.

\begin{lemma}\label{details}
Let $f$ be a meromorphic function such that $f(0)\not=0,\infty$
and let $c\in\C$. Then for all $\alpha >1$, $\delta\in(0,1)$ and
$r>0$,
    \begin{equation*}
    m\left(r,\frac{f(z+c)}{f(z)}\right) \leq
    \frac{K(\alpha,\delta,c)}{r^\delta}\left(T\big(\alpha(r+|c|),f\big)+\log^{+}\frac{1}{|f(0)|}\right),
    \end{equation*}
where
    \begin{equation*}
    K(\alpha,\delta,c)= \frac{4|c|^\delta(4\alpha+\alpha\delta+\delta)}{\delta(1-\delta)(\alpha-1)}.
    \end{equation*}
\end{lemma}

\begin{proof} By the Poisson-Jensen
formula \cite[Theorem 1.1]{hayman:64},
    \begin{equation}\label{integratethis}
    \begin{split}
    \log \left|\frac{f(z+c)}{f(z)}\right| &= \int_0^{2\pi}
    \log|f(se^{i\theta})|\textrm{Re}\left(\frac{se^{i\theta}+z+c}{se^{i\theta}-z-c}-
    \frac{se^{i\theta}+z}{se^{i\theta}-z}\right)\,\frac{d\theta}{2\pi}\\
    &\quad  + \sum_{|a_n|<s} \log \left|\frac{s(z+c-a_n)}{s^2-\bar a_n(z+c)}\cdot\frac{s^2-
    \bar    a_n z}{s(z-a_n)}\right| \\
    &\quad  - \sum_{|b_m|<s} \log \left|\frac{s(z+c-b_m)}{s^2-\bar b_m(z+c)}\cdot\frac{s^2-
    \bar    b_m z}{s(z-b_m)}\right|,
    \end{split}
    \end{equation}
where $|z|=r$, $s=\frac{\alpha+1}{2}(r+|c|)$, and $\{a_j\}$ and
$\{b_m\}$ are the sequences of zeros and poles of $f$,
respectively. Hence, by denoting $\{q_k\}:=\{a_j\}\cup\{b_m\}$ and
integrating \eqref{integratethis} over the set
$\{\xi\in[0,2\pi):
\left|\frac{f(re^{i\xi}+c)}{f(re^{i\xi})}\right|\geq 1\}$,
it follows that
    \begin{equation}\label{s1s2}
    m\left(r,\frac{f(z+c)}{f(z)}\right) \leq
     S_1(r)+S_2(r),
     \end{equation}
where
    \begin{equation}\label{S1prev}
    \begin{split}
    S_1(r) &= \int_0^{2\pi} \int_0^{2\pi}
    \left|\log|f(se^{i\theta})|\textrm{Re}\left(\frac{2c se^{i\theta}}{(se^{i\theta}-re^{i\varphi}-c)(se^{i\theta}-
    re^{i\varphi})}\right)\right|\,\frac{d\theta}{2\pi}\frac{d\varphi}{2\pi}
    \end{split}
    \end{equation}
and
    \begin{equation*}
    \begin{split}
    S_2(r) &= \sum_{|q_k|<s}
    \int_0^{2\pi}\log^{+}\left|1+\frac{c}{re^{i\varphi}-q_k}\right| \,\frac{d\varphi}{2\pi}
     \\
    & \quad +  \sum_{|q_k|<s}\int_0^{2\pi}\log^{+}\left|1-\frac{c}{re^{i\varphi}+c-q_k}\right|
    \,\frac{d\varphi}{2\pi}
    \\
    &\quad+ \sum_{|q_k|<s}  \int_0^{2\pi}\log^{+}\left|1+\frac{c}{re^{i\varphi}-\frac{s^2}{\bar q_k}
    }\right| \,\frac{d\varphi}{2\pi}\\
    & \quad+   \sum_{|q_k|<s} \int_0^{2\pi}\log^{+}\left|1-\frac{c}{re^{i\varphi}+c-\frac{s^2}{\bar q_k}
    }\right| \,\frac{d\varphi}{2\pi}.
    \end{split}
    \end{equation*}
By interchanging the order of integration in \eqref{S1prev} using Fubini's theorem, it follows that
    \begin{equation*}
    \begin{split}
    S_1(r) &= \int_0^{2\pi} \left|\log|f(se^{i\theta})|\right| \int_0^{2\pi}
    \left|\textrm{Re}\left(\frac{2c se^{i\theta}}{(se^{i\theta}-re^{i\varphi}-c)(se^{i\theta}-
    re^{i\varphi})}\right)\right|\,\frac{d\varphi}{2\pi}\frac{d\theta}{2\pi}.
    \end{split}
    \end{equation*}
Therefore, by applying the inequality \eqref{aux1} and using the identities
$s=(\alpha+1)(r+|c|)/2$, $s-r-|c|=(\alpha-1)(r+|c|)/2$ and
$s-r\geq |c|$, we have
    \begin{equation}\label{S1}
    \begin{split}
    S_1(r) &\leq \frac{2|c|s}{(s-r-|c|)(s-r)^{1-\delta}} \int_0^{2\pi} \left|\log|f(se^{i\theta})|\right|  \int_0^{2\pi}
    \frac{1}{|se^{i\theta}-
    re^{i\varphi}|^\delta}\,\frac{d\varphi}{2\pi}\frac{d\theta}{2\pi}\\
    &\leq  \frac{2|c|^\delta}{(1-\delta)}\cdot\frac{\alpha+1}{\alpha-1}\cdot\frac{1}{r^{\delta}}  \left(m(s,f)+m\left(s,\frac{1}{f}\right)\right)\\
    &\leq \frac{4|c|^\delta}{(1-\delta)}\cdot\frac{\alpha+1}{\alpha-1}\cdot\frac{1}{r^{\delta}}
    \left(T(s,f)+\log^+\frac{1}{|f(0)|}\right).
    \end{split}
    \end{equation}
Moreover, by using Lemma \ref{ineqlemma} and substituting
$s=\frac{\alpha+1}{2}(r+|c|)$, it follows that
    \begin{equation}\label{S23}
    \begin{split}
    S_2(r) &\leq \frac{4}{\delta}\left(n(s,f)+n\left(s,\frac{1}{f}\right)\right)
    \log^{+}\left(1+\frac{|c|^\delta}{(1-\delta)}\frac{1}{r^\delta}\right)\\
    &\leq \frac{16\alpha}{(\alpha-1)\delta}\left(T(\alpha(r+|c|),f)+\log^{+}\frac{1}{|f(0)|}\right)
    \log^{+}\left(1+\frac{|c|^\delta}{(1-\delta)}\frac{1}{r^\delta}\right)\\
    &\leq \frac{16|c|^\delta}{\delta(1-\delta)}\cdot\frac{\alpha}{\alpha-1}\cdot\frac{1}{r^\delta}\left(T(\alpha(r+|c|),f)+\log^{+}\frac{1}{|f(0)|}\right)
    \\
    \end{split}
    \end{equation}
The assertion follows by combining the inequalities \eqref{s1s2},
\eqref{S1} and \eqref{S23}.
\end{proof}

\noindent\textit{Proof of Theorem \ref{logder}: } If $f$ has
either a zero or a pole at the origin, then the assertion can be
proved by considering the function $w(z)=z^k f (z)$, where
$k\in\Z$ is chosen such that $w(0)\not=0,\infty$. Therefore it is
sufficient to consider only the case where $f(0)\not=0,\infty$.
Let $\xi(x)$ and $\phi(s)$ be positive, nondecreasing and
continuous functions defined for all sufficiently large $x$ and
$s$, respectively, and let $C>1$. Then, by \cite[Lemma
3.3.1]{cherryy:01}, we have
    \begin{equation}\label{cy}
    T\left(s+\frac{\phi(s)}{\xi(T(s,f))},f\right) \leq C\,T(s,f)
    \end{equation}
for all $s$ outside of a set $E$ satisfying
    \begin{equation}\label{E}
    \int_{E\cap [s_0,R]} \frac{ds}{\phi(s)} \leq  \frac{1}{\log
     C}\int_e^{T(R,f)}\frac{dx}{x\xi(x)} +O(1)
    \end{equation}
where $R<\infty$.

Assume first that $f$ is of finite order. By choosing $\phi(r)=
r$, $\xi(x)=1$ and $\alpha=2$ in Lemma~\ref{details} and
\eqref{cy}, it follows that
    \begin{equation}\label{gT}
    T(\alpha(r+|c|),f) =  T\left(r+|c|+\frac{\phi(r+|c|)}{\xi(T(r+|c|,f))},f\right) \leq C\, T(r+|c|,f)
    \end{equation}
for all $r$ outside of a set $E=E(f,\varepsilon)$ which according to \eqref{E}
satisfies
    \begin{equation}\label{ER}
    \int_{E\cap [1,R]} \frac{ds}{s} \leq  \frac{1}{\log
     C}\int_e^{R^\rho}\frac{dx}{x} +O(1)
    \end{equation}
for some $\rho>0$. Then, choosing
$C=\exp(\rho/\varepsilon)$, it follows by \eqref{ER} that
    \begin{equation*}
    d(E)=\limsup_{R\to\infty}\frac{\int_{E\cap [1,R]} \frac{ds}{s}}{\log R} \leq  \frac{\rho}{\log
     C} = \varepsilon.
    \end{equation*}
Therefore, $T(2r,f)=O(T(r,f))$ for all $r$ outside of an exceptional set $E=E(f,\varepsilon)$ of logarithmic
density $d(E)\leq \varepsilon$, and so by choosing $\delta=1-1/\log r$ in
Lemma~\ref{details} it follows that
    \begin{equation}\label{m2}
    m\left(r,\frac{f(z+c)}{f(z)}\right) =O\left(
    \frac{\log r}{r}\,T(r+|c|,f)\right)
    \end{equation}
as $r\to\infty$ such that $r\not\in E$.

If $f$ is of infinite order and $\varsigma(f)<1$, then by choosing
$\phi(r)= r$, $\xi(x)=(\log x)^{1+\varepsilon/3}$ and
    \begin{equation*}
    \alpha= 1+\frac{\phi(r+|c|)}{(r+|c|)\xi(T(r+|c|,f))},
    \end{equation*}
in Lemma~\ref{details} and \eqref{cy}, it follows that
    \begin{equation}\label{m3}
    m\left(r,\frac{f(z+c)}{f(z)}\right) =
    o\left(\frac{T(r+|c|,f)}{r^{1-\varsigma-\varepsilon}}\right)
    \end{equation}
as $r$ approaches infinity outside of an $r$-set of finite
logarithmic measure.

Asymptotic relations \eqref{m2} and
\eqref{m3} together with the following lemma, which is a
generalization of \cite[Lemma 2.1]{halburdk:07PLMS}, yield the
assertion of Theorem \ref{logder}. \hfill$\Box$

\begin{lemma}\label{technical}
Let $T:[0,+\infty)\to[0,+\infty)$ be a non-decreasing continuous
function and let $s\in(0,\infty)$. If the hyper-order of $T$ is
strictly less than one, i.e.,
    \begin{equation}\label{assu}
    \limsup_{r\to\infty}\frac{\log\log T(r)}{\log r}=\varsigma<1
    \end{equation}
and $\delta\in(0,1-\varsigma)$ then
   \begin{equation}\label{concl}
    T(r+s) = T(r)+ o\left(\frac{T(r)}{r^{\delta}}\right)
    \end{equation}
where $r$ runs to infinity outside of a set of finite logarithmic
measure.
\end{lemma}

\begin{proof}
Let $\tilde\delta\in(\delta,1-\varsigma)$, $\eta\in\mathbb{R}^+$
and assume that the set $F_{\eta}\subset[1,\infty)$ defined by
    \begin{equation}\label{Fdef}
    F_{\eta}=\left\{r\in\mathbb{R}^+:
    \frac{T(r+s)-T(r)}{T(r)}\cdot r^{\tilde\delta}\geq \eta
    \right\}
    \end{equation}
is of infinite logarithmic measure. Note that $F_\eta$ is a closed
set and therefore it has a smallest element, say $r_0$. Set
$r_n=\min\{F_\eta\cap [r_{n-1}+s,\infty)\}$ for all $n\in\N$. Then
the sequence $\{r_n\}_{n\in\Z^+}$ satisfies $r_{n+1}-r_n\geq s$
for all $n\in\Z^+$, $F_\eta\subset \bigcup_{n=0}^\infty
[r_n,r_n+s]$ and
    \begin{equation}\label{assuinpr2}
    \left(1+\frac{\eta}{{r_n}^{\tilde\delta}}\right)T(r_n)\leq T(r_{n+1})
    \end{equation}
for all $n\in\Z^+$.

Let $\varepsilon>0$, and suppose that there exists an $m\in\Z^+$
such that $r_n\geq n^{1+\varepsilon}$ for all $r_n\geq m$. But
then,
    \begin{eqnarray*}
    \int_{F_\eta\cap[1,\infty)}\frac{dt}{t} &\leq& \sum_{n=0}^\infty
    \int_{r_n}^{r_n+s}\frac{dt}{t}
    \leq \int_1^{m} \frac{dt}{t} +  \sum_{n=1}^\infty
    \log\left(1+\frac{s}{r_n}\right)\\
    &\leq& \sum_{n=1}^\infty
    \log\left(1+s n^{-(1+\varepsilon)}\right) +O(1)  <\infty
    \end{eqnarray*}
which contradicts the assumption $\int_{F_\eta\cap[1,\infty)}\frac{dt}{t}=\infty$.
Therefore the sequence $\{r_n\}_{n\in\Z^+}$ has a subsequence
$\{r_{n_j}\}_{j\in\Z^+}$ such that $r_{n_j}\leq
n_j^{1+\varepsilon}$ for all $j\in\Z^+$. By
iterating~(\ref{assuinpr2}) along the sequence
$\{r_{n_j}\}_{j\in\Z^+}$, we have
    \begin{equation*}
    T(r_{n_j})\geq \prod_{\nu=0}^{n_j-1}\left(1+\frac{\eta}{{r_{\nu}}^{\tilde\delta}}\right)T(r_0)
    \end{equation*}
for all $j\in \Z^+$, and so
    \begin{equation*}
    \begin{split}
    \limsup_{r\to\infty}\frac{\log\log T(r)}{\log r}
    &\geq
    \limsup_{j\to\infty}\frac{\displaystyle\log\left(\log T(r_0)+\sum_{\nu=0}^{n_j-1}\log\left(1+\frac{\eta}{r_\nu^{\tilde\delta}}\right)\right)}{\log r_{n_j}}\\
    &\geq\limsup_{j\to\infty}\frac{\displaystyle\log\left(\log
    T(r_0)+n_j\log\left(1+\frac{\eta}{r_{n_j}^{\tilde\delta}}\right)\right)}
    {(1+\varepsilon)\log n_j}\\
    &\geq \limsup_{j\to\infty}\frac{\displaystyle\log\Bigg(\log T(r_0)+n_j
    \frac{\eta}{n_j^{(1+\varepsilon)\tilde\delta}}
    \log\Bigg(1+\frac{\eta}{n_j^{(1+\varepsilon)\tilde\delta}}\Bigg)^{n_j^{\frac{(1+\varepsilon)\tilde\delta}{\eta}}}\Bigg)}
    {(1+\varepsilon)\log n_j}\\
    &\geq
    \limsup_{j\to\infty}\frac{[1-(1+\varepsilon)\tilde\delta] \log n_j}{(1+\varepsilon)\log n_j}\\
    &\geq \frac{1}{1+\varepsilon}-\tilde\delta.
    \end{split}
    \end{equation*}
By letting $\varepsilon\to0$, we obtain
    \begin{equation*}
    \limsup_{r\to\infty}\frac{\log\log T(r)}{\log
    r}\geq1-\tilde\delta
    \end{equation*}
which contradicts \eqref{assu} since $1-\tilde\delta>\varsigma$. Hence the logarithmic measure of
$F_\eta$ defined by \eqref{Fdef} must be finite, and so
    \begin{equation*}
    T(r+s) = T(r)+ O\left(\frac{T(r)}{r^{\tilde\delta}}\right)
    \end{equation*}
for all $r$ outside of a set of finite logarithmic measure.
Therefore the assertion \eqref{concl} follows.
\end{proof}

\section{The proof of Theorem \ref{casorati}}\label{section_proof}

The following lemma is due to Cartan \cite{cartan:33} (see also
\cite{gundersenh:04}).

\begin{lemma}[\cite{cartan:33}]\label{cartanlemma}
Let $n\geq 1$, let $z\in\C$ and let $g_0,\ldots,g_n$ be linearly
independent entire functions such that
$\max\{|g_0(z)|,\ldots,|g_n(z)|\}>0$ for each $z\in\C$. If
$f_0,\ldots,f_q$ are $q+1$ linear combinations of the $n+1$
functions $g_0,\ldots,g_n$, where $q>n$, such that any $n+1$ of
the $q+1$ functions $f_0,\ldots,f_q$ are linearly independent,
then there exists a positive constant $A$ that does not depend on
$z$, such that
    \begin{equation*}
    |g_j(z)|\leq A|f_{m_\nu}(z)|,
    \end{equation*}
where $0\leq j \leq n$, $0\leq\nu\leq q-n$ and the integers
$m_0,\ldots,m_q$ are chosen so that
    \begin{equation*}
    |f_{m_0}(z)|\geq |f_{m_2}(z)|\geq \cdots \geq |f_{m_q}(z)|.
    \end{equation*}
In particular, there exist at least $q-n+1$ functions $f_j$ that
do not vanish at $z$.
\end{lemma}

\noindent\textit{Proof of Theorem \ref{casorati}: } The proof
follows the original proof of Cartan's second main theorem, see,
e.g., \cite{cartan:33,hayman:84,gundersenh:04}, taking into
account the special properties of the Casorati determinant. Since
the functions $g_j$, where $j=0,\ldots,n$, are linearly independent over $\mathcal{P}^1_c$, it follows by Lemma~\ref{casoratilemma} that $C(g_0,\ldots,g_n)\not\equiv0$ and so the function $L$ is well defined. The functions $g_j$, $j=0,\ldots,n$, are also linearly independent over $\C$ (since $\C\subset\mathcal{P}^1_c$), and so by Lemma \ref{cartanlemma} the auxiliary
function
    \begin{equation}\label{v}
    v(z)=\max_{\{k_j\}_{j=0}^{q-n-1}\subset\{0,\ldots,q\}}\log|f_{k_0}(z)\cdots f_{k_{q-n-1}}(z)|
    \end{equation}
gives a finite real number for all $z\in\C$. Let
$\{a_0,\ldots,a_{q-n-1}\}\subset \{0,\ldots,q\}$, and
$\{b_0,\ldots,b_n\}=\{0,\ldots,q\}\setminus
\{a_0,\ldots,a_{q-n-1}\}$. Since $f_{b_0},\ldots,f_{b_n}$ are
linearly independent linear combinations of $g_0,\ldots,g_n$, it follows that
$C(f_{b_0},\ldots,f_{b_n})\not\equiv0$, and
    \begin{equation*}
    \left(\begin{array}{cccc}
      f_{b_0} & \cdots & f_{b_n} \\
      \overline{f}_{b_0}  & \cdots & \overline{f}_{b_n} \\
      \vdots & \ddots & \vdots \\
      \overline{f}^{[n]}_{b_0} &  \cdots & \overline{f}^{[n]}_{b_n} \\
    \end{array}\right)=\left(\begin{array}{cccc}
      g_0 &  \cdots & g_n \\
      \overline{g}_0 & \cdots & \overline{g}_n \\
      \vdots  & \ddots & \vdots \\
      \overline{g}^{[n]}_0 &  \cdots & \overline{g}^{[n]}_n \\
    \end{array}\right)\left(\begin{array}{cccc}
      \pi_{00} &  \cdots & \pi_{0n} \\
      \pi_{10} &  \cdots & \pi_{1n}\\
      \vdots &  \ddots & \vdots \\
      \pi_{n0} &   \cdots & \pi_{nn} \\
    \end{array}\right)
    \end{equation*}
where $\pi_{jm}\in \C$ for all $j=0,\ldots,n$ and $m=0,\ldots,n$.
Therefore,
    \begin{equation}\label{C}
    C(g_0,\ldots,g_n)=A(b_0,\ldots,b_n) C(f_{b_0},\ldots,f_{b_n})
    \end{equation}
where $A(b_0,\ldots,b_n)=:A_\textbf{b}\in \C\setminus\{0\}$.

We prove the inequality \eqref{casoratiineq} first for the auxiliary function
    \begin{equation}\label{Ltilde}
    \widetilde L := \frac{f_0\overline{f}_1\cdots \overline{f}^{[n]}_n f_{n+1}\cdots
    f_q}{C(g_0,\ldots,g_n)},
    \end{equation}
which is also well defined since $C(g_0,\ldots,g_n)\not\equiv0$. By substituting \eqref{C} into
\eqref{Ltilde}, we have
    \begin{equation*}
    \begin{split}
    \widetilde L&=\frac{f_0\overline{f}_1\cdots \overline{f}^{[n]}_n f_{n+1}\cdots
    f_q}{A_\textbf{b} C(f_{b_0},f_{b_1},\ldots,f_{b_n})}\\
    &= \frac{ f_0\cdots f_q \cdot (\overline{f}_1/f_1)\cdots (\overline{f}_n^{[n]}/f_n)}
    {A_\textbf{b} C(f_{b_0},f_{b_1},\ldots,f_{b_n})}\\
    &=\frac{ f_{b_0}\overline{f}_{b_1}\cdots \overline{f}_{b_n}^{[n]} \cdot
    f_{a_0}\cdots f_{a_{q-n-1}} (\overline{f}_1/f_1)\cdots
    (\overline{f}_n^{[n]}/f_n)\cdot(f_{b_1}/\overline{f}_{b_1})\cdots(f_{b_n}/\overline{f}_{b_n}^{[n]})}
    {A_\textbf{b} C(f_{b_0},f_{b_1},\ldots,f_{b_n})}\\
    &= \frac{
    f_{a_0}\cdots f_{a_{q-n-1}} (\overline{f}_1/f_1)\cdot (f_{b_1}/\overline{f}_{b_1})\cdots
    (\overline{f}_n^{[n]}/f_n)\cdot(f_{b_n}/\overline{f}_{b_n}^{[n]})}
    {\displaystyle \left( \frac{A_\textbf{b} f_{0}\overline{f}_{0}\cdots \overline{f}_{0}^{[n]}
    C(f_{b_0}/f_0,f_{b_1}/f_0,\ldots,f_{b_n}/f_0)}{f_{b_0}\overline{f}_{b_1}\cdots \overline{f}_{b_n}^{[n]}}\right)}\\
    &= \frac{
    f_{a_0}\cdots f_{a_{q-n-1}} (\overline{f}_1/\overline{f}_{b_1})/(f_1/f_{b_1})\cdots
    (\overline{f}_n^{[n]}/\overline{f}_{b_n}^{[n]})/(f_n/f_{b_n})}
    {\displaystyle \left(\frac{A_\textbf{b} f_{0}\overline{f}_{0}\cdots \overline{f}_{0}^{[n]}
    C(f_{b_0}/f_0,f_{b_1}/f_0,\ldots,f_{b_n}/f_0)}{f_{b_0}\overline{f}_{b_1}\cdots \overline{f}_{b_n}^{[n]}}\right)}\\
    &= \frac{
    f_{a_0}\cdots f_{a_{q-n-1}} (\overline{f}_1/\overline{f}_{b_1})/(f_1/f_{b_1})\cdots
    (\overline{f}_n^{[n]}/\overline{f}_{b_n}^{[n]})/(f_n/f_{b_n})}
    {\displaystyle \left( \frac{A_\textbf{b} C(f_{b_0}/f_0,f_{b_1}/f_0,\ldots,f_{b_n}/f_0)}{(f_{b_0}/f_0)\cdot
    (\overline{f}_{b_1}/\overline{f}_{0})\cdots (\overline{f}_{b_n}^{[n]}/\overline{f}_{0}^{[n]})}\right)}.\\
    \end{split}
    \end{equation*}
Therefore,
    \begin{equation*}
    \widetilde L=\frac{f_{a_0}\cdots f_{a_{q-n-1}}}{A_\textbf{b} G}
    \end{equation*}
where
    \begin{equation}\label{G}
    G=\frac{
    {\displaystyle \left( \frac{ C(f_{b_0}/f_0,f_{b_1}/f_0,\ldots,f_{b_n}/f_0)}{(f_{b_0}/f_0)\cdot
    (\overline{f}_{b_1}/\overline{f}_{0})\cdots (\overline{f}_{b_n}^{[n]}/\overline{f}_{0}^{[n]})}\right)}}{(\overline{f}_1/\overline{f}_{b_1})/(f_1/f_{b_1})\cdots
    (\overline{f}_n^{[n]}/\overline{f}_{b_n}^{[n]})/(f_n/f_{b_n})}.
    \end{equation}
By defining
    \begin{equation*}
    w(z)=\max_{\{b_j\}_{j=0}^n\subset\{0,\ldots,q\}}\log|A_\textbf{b} G(z)|
    \end{equation*}
it follows that $v(z)=\log|\widetilde L(z)|+w(z)$ whenever $\widetilde L(z)$ non-zero
and finite, and so
    \begin{equation}\label{intH}
    \int_0^{2\pi} v(re^{i\theta})d\theta =  \int_0^{2\pi}
    \log|\widetilde L(re^{i\theta})|d\theta +     \int_0^{2\pi} w(re^{i\theta})d\theta.
    \end{equation}
(If $\widetilde L$ has zeros or poles on the circle $\{z:|z|=r\}$, then the path of integration may be slightly
amended in \eqref{intH} so that any poles or zeros of $\widetilde L$ are avoided by the new
path. The validity of \eqref{intH} then follows by a limiting
argument where the modified path is allowed to approach the circle
$\{z:|z|=r\}$. See, e.g., \cite{gundersenh:04} for more details.)

Let $\{c_0,\ldots,c_{q-n-1}\}$ be the set of indexes for which the
maximum in \eqref{v} is attained for a particular choice of
$z\in\C$. Then by Lemma \ref{cartanlemma} it follows that
$\log|g_j(z)|\leq \log|f_{c_\nu}(z)|+\log A$ for all $0\leq j\leq
n$ and $0\leq \nu \leq q-n-1$, and so
    \begin{equation}\label{T}
    (q-n)T_g(r)\leq
    \frac{1}{2\pi}\int_0^{2\pi}v(re^{i\theta})d\theta + O(1)
    \end{equation}
as $r\to\infty$. Since the function $G$ in \eqref{G} consists purely of sums,
products and quotients of fractions of the form
$(\overline{f}_{j}^{[l]}/\overline{f}_k^{[l]})/(f_j/f_k)$ where
$l\in\{1,\ldots,n\}$ and $j,k\in\{0,\ldots,q\}$, it follows by
Theorem \ref{logder} that
    \begin{equation}\label{logdappl}
    \frac{1}{2\pi}\int_0^{2\pi}w(re^{i\theta})d\theta \leq   \sum_{j=0}^q\sum_{k=0}^q
    o\left(\frac{T(r,f_j/f_k)}{r^{1-\varsigma-\varepsilon}}\right)+O(1)
    \end{equation}
as $r$ approaches infinity outside of an exceptional set of finite
logarithmic measure. By combining \eqref{logdappl} and
\eqref{cartanT} it follows that
    \begin{equation}\label{westim}
    \frac{1}{2\pi}\int_0^{2\pi}w(re^{i\theta})d\theta =o\left(\frac{T_g(r)}{r^{1-\varsigma-\varepsilon}}\right)+O(1)
    \end{equation}
where $r$ tends to infinity outside of an exceptional set of
finite logarithmic measure. Finally, by Jensen's formula,
    \begin{equation}\label{jensen}
    \frac{1}{2\pi}\int_0^{2\pi} \log|\widetilde L(re^{i\theta})|d\theta =
    N\left(r,\frac{1}{\widetilde L}\right)-N(r,\widetilde L) + O(1)
    \end{equation}
as $r\to\infty$, and therefore we have
    \begin{equation}\label{tildeLineq}
    (q-n)T_g(r)\leq
    N\left(r,\frac{1}{\widetilde L}\right)-N(r,\widetilde L)+o\left(\frac{T_g(r)}{r^{1-\varsigma-\varepsilon}}\right)+O(1)
    \end{equation}
by combining
\eqref{intH}, \eqref{T}, \eqref{westim} and \eqref{jensen}.

We will complete the proof by showing how \eqref{casoratiineq} follows from \eqref{tildeLineq}. Consider first the counting functions
    \begin{equation}\label{Nj}
    N\left(r,\frac{1}{\overline{f}_j^{[j]}}\right) \leq N\left(r+j,\frac{1}{f_j}\right)
    \end{equation}
for $j=1,\ldots,n$. In order to apply Lemma \ref{technical} to the right side of inequality \eqref{Nj} we need to consider the growth of $N(r,1/f_j)$. Since each $f_j$ is an entire linear combination of $g_0,\ldots,g_n$, we have by Poisson-Jensen formula that
    \begin{equation}\label{Njgrowth}
    \begin{split}
    N\left(r,\frac{1}{f_j}\right)&=\int_0^{2\pi}\log|f_j(re^{i\theta})|\frac{d\theta}{2\pi}\\
    &\leq K \int_0^{2\pi}\sup_{j=0,\ldots,n}\log|g_j(re^{i\theta})|\frac{d\theta}{2\pi}\\
    &=K T_g(r)+O(1),
    \end{split}
    \end{equation}
where the constant $K>0$ is independent of $r$. Since $\varsigma(g)<1$, it follows by \eqref{Njgrowth} that
    $$
    \delta_j := \limsup_{r\to\infty} \frac{\log\log N\left(r,\frac{1}{f_j}\right)}{\log r}<1
    $$
for all $j=1,\ldots,n$. Therefore, by Lemma \ref{technical}, we have
    \begin{equation}\label{Njgrowth2}
    N\left(r,\frac{1}{\overline{f}_j^{[j]}}\right) \leq N\left(r+j,\frac{1}{f_j}\right) = N\left(r,\frac{1}{f_j}\right) + o\left(\frac{N\left(r,\frac{1}{f_j}\right)}{r^{1-\delta_j-\varepsilon}}\right),
    \end{equation}
where $j=1,\ldots,n$ and $r$ tends to infinity outside of an exceptional set of finite logarithmic density.  By using \eqref{Njgrowth}, the inequality \eqref{Njgrowth2} yields
    \begin{equation*}
    N\left(r,\frac{1}{\overline{f}_j^{[j]}}\right)\leq  N\left(r,\frac{1}{f_j}\right)     + o\left(\frac{T_g(r)}{r^{1-\varsigma-\varepsilon}}\right), \qquad j=1,\ldots,n,
    \end{equation*}
outside of an exceptional set of finite logarithmic measure. Therefore,
    \begin{equation*}
    \begin{split}
    &N\left(r,\frac{1}{\widetilde L}\right)-N(r,\widetilde L) \\ &\quad=  N\left(r,\frac{C(g_0,\ldots,g_n)}{f_0\overline{f}_1\cdots \overline{f}^{[n]}_n f_{n+1}\cdots
    f_q}\right)-N\left(r,\frac{f_0\overline{f}_1\cdots \overline{f}^{[n]}_n f_{n+1}\cdots
    f_q}{C(g_0,\ldots,g_n)}\right) \\ &\quad= N\left(r,\frac{1}{f_0\overline{f}_1\cdots \overline{f}^{[n]}_n f_{n+1}\cdots
    f_q}\right) - N\left(r,\frac{1}{C(g_0,\ldots,g_n)}\right) \\
    &\quad = \sum_{j=0}^n N\left(r,\frac{1}{\overline{f}^{[j]}_j}\right) + N\left(r,\frac{1}{f_{n+1}\cdots
    f_q}\right) - N\left(r,\frac{1}{C(g_0,\ldots,g_n)}\right)\\
    &\quad \leq  \sum_{j=0}^n N\left(r,\frac{1}{f_j}\right)     + N\left(r,\frac{1}{f_{n+1}\cdots
    f_q}\right) - N\left(r,\frac{1}{C(g_0,\ldots,g_n)}\right) + o\left(\frac{T_g(r)}{r^{1-\varsigma-\varepsilon}}\right)
    \\
    &\quad = N\left(r,\frac{1}{f_0\cdots f_q}\right) - N\left(r,\frac{1}{C(g_0,\ldots,g_n)}\right) + o\left(\frac{T_g(r)}{r^{1-\varsigma-\varepsilon}}\right)\\
    &\quad = N\left(r,\frac{1}{L}\right)-N(r,L) + o\left(\frac{T_g(r)}{r^{1-\varsigma-\varepsilon}}\right).
    \end{split}
    \end{equation*}
The assertion follows by substituting this inequality into \eqref{tildeLineq}. \hfill$\Box$

\section{The proof of Theorem \ref{hypers}}\label{section_proofh}

Let $x=[x_0:\cdots:x_n]$, and let $H_j(x)$ be the linear form
defining the hyperplane $H_j(x)=0$ for all $j=1,\ldots,n+p$. Since
by assumption any $n+1$ of the hyperplanes $H_j$,
$j=1,\ldots,n+p$, are linearly independent, it follows
that any $n+2$ of the forms $H_j(x)$ satisfy a linear relation
with non-zero coefficients in $\C$. By writing $\tau(z)=z+c$ and
$f=[f_0:\ldots:f_n]$, where the coordinate functions are entire functions without
common zeros, it follows by assumption that the functions $h_j=H_j(f)$ satisfy
    \begin{equation}\label{shift}
   \left\{ \tau(h_j^{-1}(\{0\}))\right\}\subset
   \left\{h_j^{-1}(\{0\})\right\}
    \end{equation}
for all $j=1,\ldots,n+p$, where $\{\cdot\}$ denotes a multiset which takes into account the multiplicities of its elements. The
set of indexes $\{1,\ldots,n+p\}$ may be split into disjoint
equivalence classes $S_k$ by saying that $i\sim j$ if $h_i=\alpha
h_j$ for some $\alpha\in \mathcal{P}_c^1\setminus\{0\}$. Therefore
    \begin{equation*}
    \{1,\ldots,n+p\}=\bigcup_{j=1}^N S_j
    \end{equation*}
for some $N\in\{1,\ldots,n+p\}$.

Suppose that the complement of $S_k$ has at least $n+1$ elements
for some $k\in\{1,\ldots,N\}$. Choose an element $s_0\in S_k$, and
denote $U=\{1,\ldots,n+p\}\setminus S_k\cup\{s_0\}$. Since the set
$U$ contains at least $n+2$ elements, there exists a subset
$U_0\subset U$ such that $U_0\cap S_k=\{s_0\}$ and
$\card(U_0)=n+2$. Therefore, there exists $\alpha_j\in
\C\setminus\{0\}$ such that
    \begin{equation*}
    \sum_{j\in U_0}\alpha_j H_j =0,
    \end{equation*}
and so
    \begin{equation*}
    \sum_{j\in U_0} \alpha_j h_j =0.
    \end{equation*}
This contradicts Theorem~\ref{borelanalogue}, and so the set
$\{1,\ldots,n+p\}\setminus S_k$ has at most $n$ elements. Hence
$S_k$ has at least $p$ elements for all $k=1,\ldots,N$, and it
follows that $N\leq (n+p)/p$.

Let $V$ be any subset of $\{1,\ldots,n+p\}$ with exactly $n+1$
elements. Then the forms $H_j$, $j\in V$, are linearly
independent. By denoting $V_k=V\cap S_k$ it follows that
    \begin{equation*}
    V=\bigcup_{k=1}^N V_k.
    \end{equation*}
Since each set $V_k$ gives raise to $\card(V_k)-1$ equations over the field $\mathcal{P}_c^1$, it
follows that we have at least
    \begin{equation*}
    \sum_{j=1}^N \card(V_k)-1 = n+1-N \geq
    n+1-\frac{n+p}{p}=n-\frac{n}{p}
    \end{equation*}
linearly independent relations over the field $\mathcal{P}_c^1$. Therefore the
image of $f$ is contained in a linear subspace over $\mathcal{P}_c^1$ of
dimension $\leq[n/p]$, as desired. \hfill$\Box$

\medskip

\noindent \textbf{Acknowledgments. } We wish to express our gratitude to Professor N. Toda for his critical comments. We thank the referee for his/her valuable suggestions which helped us to improve the manuscript, and for bringing the reference \cite{wonglw:09} to our attention.

\bibliographystyle{amsplain}

\def\cprime{$'$}
\providecommand{\bysame}{\leavevmode\hbox to3em{\hrulefill}\thinspace}
\providecommand{\MR}{\relax\ifhmode\unskip\space\fi MR }
\providecommand{\MRhref}[2]{%
  \href{http://www.ams.org/mathscinet-getitem?mr=#1}{#2}
}
\providecommand{\href}[2]{#2}

\end{document}